\documentclass[11pt]{article}
\usepackage{amsmath,amssymb,amsthm,amsfonts,amstext,amsbsy,amscd}

\newtheorem{lem}{Lemma}
\newtheorem{theo}{Theorem}
\newtheorem{rem}{Remark}
\newtheorem{prop}{Proposition}
\newtheorem{coro}{Corollary}

\usepackage{amsmath}
\usepackage{amssymb}
\usepackage{graphicx}

\usepackage{color}
\newcommand{\abs}[1]{\left| #1\right|}

\newcommand{\ind}{ 1 }
\newcommand{\exspace}{{\mathcal{C}_{0\to 0}}}
\newcommand{\dd}{\text{{\rm d}}}
\newcommand{\ds}{\displaystyle}

\begin{document}
\author{
Arnaud Gloter\thanks{ Universit\'e d'\'Evry Val d'Essonne, D\'epartement de Math\'ematiques, 91025 \'Evry Cedex, France.}
\thanks{ 
This research benefited from the support of the 'Chair Risque de cr\'edit', F\'ed\'eration Bancaire Fran\c{c}aise. }
\and
Miguel Martinez \thanks{Universit\'e Paris-Est, Laboratoire d'Analyse et de
Math\'ematiques Appliqu\'ees, UMR $8050$, 5  Bld Descartes,
Champs-sur-marne, 77454 Marne-la-Vall\'ee Cedex 2, France. } 
}
\title{Distance between two skew Brownian motions as a S.D.E. 
with jumps
and law of the hitting time}
\maketitle

\begin{abstract}
In this paper, we consider two skew Brownian motions, driven by the same Brownian motion, with different starting points and different skewness coefficients. We show that
we can describe the evolution of the distance between the two processes with a stochastic differential equation.
This S.D.E. possesses a jump component driven by the excursion process of one of the two skew Brownian motions. 
Using this representation, we show that the local time of two skew Brownian motions at their first hitting
time is distributed as a simple function of a Beta random variable.
 This extends a result by 
Burdzy and Chen \cite{BurChe01}, where the law of coalescence of two skew Brownian motions with the same skewness coefficient is computed.

\textbf{MSC $2000$}. Primary: 60H10, Secondary: 60J55 60J65.

\textbf{Key words}: Skew Brownian motion; Local time; Excursion process; Dynkin's formula.
\end{abstract}


\section{Presentation of the problem}
Consider $(B_t)_{t \ge 0}$ a standard Brownian motion on some filtered probability space
$(\Omega,\mathcal{F}, (\mathcal{F}_t)_{t \ge 0}, \mathbb{P})$  where the filtration satisfies the usual right continuity and completeness conditions. Recall that the skew Brownian motion $X^{x,\beta}$
is defined as the solution of the stochastic differential equation with singular drift coefficient,
\begin{equation}
\label{E:def_skew}
X^{x,\beta}_t = x+ B_t + \beta L^0_t(X^{x,\beta}),
\end{equation}
where $\beta \in (-1,1)$ is the skewness parameter, $x \in \mathbb{R}$, 
and $L^0_t(X^{x,\beta})$ is the symmetric local time at $0$:
\begin{equation*}
L^0_t(X^{x,\beta})= \lim_{\varepsilon \to 0} \frac{1}{2 \varepsilon}
\int_0^t \ind_{ [-\varepsilon , \varepsilon ]}(X_s^{x,\beta})  d s.
\end{equation*}
It is known that a strong solution of the equation \eqref{E:def_skew} exists, 
and pathwise uniqueness holds as well 
(see \cite{BasChe05}, \cite{HarShe81}). 
Remark that in \cite{BurChe01} it is shown that $X^{x,\beta}$ can be obtained as the limit
of diffusion processes $X^{x,\beta,n}$ with smooth coefficients. Indeed, if one mollifies the singularity
due to the local time, the following diffusion processes can be defined,
\begin{equation*}
X^{x,\beta,n}_t = x + B_t + \frac{1}{2} \log (\frac{1+\beta}{1-\beta}) \int_0^t n \phi ( n X^{x,\beta,n}_s) d s,
\end{equation*}
where $\phi$ is any symmetric positive function with support on $[-1/2, 1/2]$ and having unit mass.
Then, the almost sure convergence of some sub-sequence $X^{x,\beta,n_k}$ to $X^{x,\beta}$ is shown in \cite{BurChe01}.

The skew Brownian motion is an example of a process partially reflected at some frontier. It finds applications in the fields of stochastic modelisation and of numerical simulations, especially as it is deeply connected to diffusion processes with non-continuous coefficients (see \cite{Lejay06} and references therein). The structure of the flow of a reflected, or partially reflected, Brownian motion has been the subject of several works (see e.g. \cite{BarBur_etal01}, \cite{Burdzy09}). The long time behaviour of the distance between reflected Brownian motions
with different starting points has been largely studied too (see e.g. \cite{BurCheJon06}, \cite{CraLej89}).

Actually, a quite intriguing fact about solutions of \eqref{E:def_skew} is that they do not satisfy the
usual flow property of differential equations, which prevents two solutions with different initial positions to meet in finite time. Indeed, it is shown in \cite{BarBur_etal01} that, almost surely, the two paths 
$t \mapsto X^{x,\beta}_t$ and $t \mapsto X^{0,\beta}_t$ meet at a finite random time. Moreover, the law of
 the values of the local times of these processes at this instant of coalescence are computed in \cite{BurChe01}. 
 
In this paper, we study the time dynamic of the distance between the two processes $X^{0,\beta_1}$ and $X^{x,\beta_2}$
where the skewness parameters $\beta_1$, $\beta_2$ are possibly different.
We show that, after some random time change, the distance between the two processes 
is a Makov process, solution to
an explicit stochastic differential equation with jumps (see Theorem \ref{T:SDEJump_main} below).  
The dynamic of this stochastic differential equation enables us to compute the law of the hitting time of zero for the distance between the two skew Brownian motions. Consequently, we can 
draw informations about the hitting time of the two skew Brownian motions.

More precisely, let us denote $T^\star$ the first
instant where $X^{x,\beta_2}$ and $X^{0,\beta_1}$ meet and define
the quantity 
$U^\star=L^{0}_{T^\star}(X^{0,\beta_1})$.
For $x>0$, $0< \beta_1,\beta_2 <1$, we show, in Theorem \ref{T:law_U} below, that the random variable
$\frac{x}{\beta_1 U^{\star}}$ is distributed with a Beta law.
This extends the result of \cite{BurChe01} where the law of the hitting time was computed 
under the restriction $\beta_1=\beta_2$.
We study also the situation where $-1<\beta_2<0<\beta_1<1$ and $x>0$. In this case, we show that the random variable $\frac{\beta_1 U^{\star}}{x}$ is distributed with a Beta law (Theorem \ref{T:law_U_diff_sign}).

The organization of the paper is as follows. In Section \ref{S:Main}, we precisely state our main results.

 
The sections \ref{S:SDE_with_jumps} and \ref{S:Hitting} are devoted to the proofs of the results in the case $0<\beta_1,\beta_2<1$. 
 In Section \ref{S:SDE_with_jumps}, we introduce our fundamental tool, which is the process
$u\mapsto X^{x,\beta_2}_{\tau_u^0 (X^{0,\beta_1})}$, where $\tau_u^0 (X^{0,\beta_1})$ is the inverse local time of $ X^{0,\beta_1}$. This process is a measurement of the distance between
$X^{0,\beta_1}$ and $X^{x,\beta_2}$. 
We prove that this process is solution of some explicit stochastic differential equation with jumps, driven by the Poisson process of the excursions of $X^{0,\beta_1}$. In Section \ref{S:Hitting}, we show how the dynamic of this process enables us to compute the law of the hitting time of the two skew Brownian motions. 

In Section \ref{S:negative_beta}, we sketch 
the proofs of our results in the situation $-1<\beta_2<0<\beta_1<1$. 
For the sake of shortness, we will only put the emphasis on the main differences with the case $0<\beta_1,\beta_2<1$.
 

\section{Main results}\label{S:Main}
Consider the two skew Brownian motions,
\begin{align}
\label{E:defXx}
 X_t^{x,\beta_2}= x + B_t + \beta_2 L_t^0(X_t^{x,\beta_2}),
 \\
 \label{E:defX0}
 X_t^{0,\beta_1}=   B_t + \beta_1 L_t^0(X_t^{0,\beta_1}),
\end{align}
with $x>0$.
We introduce the c.a.d.l.a.g. process defined as
\begin{equation} \label{E:defZ}
Z^{x,\beta_1,\beta_2}_u= X^{x,\beta_2}_{\tau_u(X^{0,\beta_1})},
\end{equation}
where $\tau_u(X^{0,\beta_1})$ is the inverse of the local time, given as,
\begin{equation*}
\tau_u(X^{0,\beta_1}) = \inf \{ t \ge 0 \mid L_t^{0} (X^{0,\beta_1}) > u \}.
\end{equation*}
Note that, since $X^{0,\beta_1}_{\tau_u(X^{0,\beta_1})}=0$, we have 
$Z^{x,\beta_1,\beta_2}_u=X^{x,\beta_2}_{\tau_u(X^{0,\beta_1})}
-X^{0,\beta_1}_{\tau_u(X^{0,\beta_1})}$.
This explains why we choose below to call $Z^{x,\beta_1,\beta_2}$ the ``distance process''.
Our first result shows that the ``distance process'' is solution to a stochastic differential equation with jumps, driven by the excursion Poisson process of $X^{0,\beta_1}$. 
We need some additional notations before stating it. 
We introduce $(\textbf{e}_u)_{u >0}$ 
the excursion process associated to $X^{0,\beta_1}$,
\begin{equation*}
\textbf{e}_u(r)=X^{0,\beta_1}_{\tau_{u-}(X^{0,\beta_1}) + r}, \quad 
\text{ for $r \le \tau_{u}(X^{0,\beta_1}) - \tau_{u-}(X^{0,\beta_1}) $.}
\end{equation*}
The Poisson point process $(\textbf{e}_u)_{u >0}$ takes values in the space $\exspace$
of excursions with finite lifetime, endowed with the usual uniform topology. We denote $\textbf{n}_{\beta_1}$ the excursion measure associated to $X^{0,\beta_1}$.

Let us define $T^\star = \inf\{ t \ge 0 \mid X_t^{0,\beta_1}=X_t^{0,\beta_2} \} \in [0,\infty]$ and $U^\star = L_{T^\star}^0(X_t^{0,\beta_1})$. 
Since $X^{x,\beta_2}$ and $X^{0,\beta_1}$ are driven by the same Brownian motion, it is easy to see that they can only meet when $X^{0,\beta_1}=0$. 
As a consequence, we have
\begin{equation*}
U^\star= \inf \{ u \ge 0 \mid Z^{x,\beta_1,\beta_2}_u=0 \} \in [0,\infty],~ \text{ and }  Z^{x,\beta_1,\beta_2}>0 \text{ on } [0,U^\star).
\end{equation*}
Our first result about $Z^{x,\beta_1,\beta_2}$ is the following.
\begin{theo}  \label{T:SDEJump_main}
Assume $x>0$ and $0<\beta_1,\beta_2<1$. Almost surely, we have for all $t < U^\star$,
\begin{equation*}
Z_t^{x,\beta_1,\beta_2} = x-\beta_1 t + \sum_{0<u\le t} \beta_2 \ell(Z_{u-}^{x,\beta_1,\beta_2},{\rm{\bf e}}_u),
\end{equation*}
where $\ell: (0,\infty) \times \exspace \to [0,\infty)$ is a measurable map. 

For $h>0$, we can describe the law of ${\rm{\bf e}} \mapsto \ell(h,{\rm{\bf e}})$ under ${\rm{\bf n}}_{\beta_1}$ by
\begin{equation} \label{E:law_jumps_main} 
{\rm{\bf n}}_{\beta_1} ( \ell(h,{\rm{\bf e}}) \ge a) = 
\frac{1-\beta_1}{2h}
\left( 1+ \frac{\beta_2 a}{h} \right)^{-\frac{1+\beta_2}{2\beta_2}}, \quad 
\forall a >0.
\end{equation} 
\end{theo}
 
\begin{rem}
Theorem \ref{T:SDEJump_main} fully details the dynamic of the ``distance process'' before it (possibly) reaches $0$. The ``distance process'' decreases with a constant negative drift, and has positive jumps. 
Moreover, the value of a jump at time $u$ is a function of the level $Z_{u-}^{x,\beta_1,\beta_2}$ and of 
the excursion $\textbf{e}_u$. 
The image of the excursion measure under this function, with a fixed level $h>0$, is given by the explicit expression \eqref{E:law_jumps_main}.
\end{rem}
In \cite{BarBur_etal01} \cite{BurChe01} it is shown that the processes $X^{0,\beta_1}$ and $X^{x,\beta_2}$ meet in finite time under some 
appropriate conditions for the skewness coefficients. 
\begin{theo}
[\cite{BarBur_etal01} \cite{BurChe01}]
\label{T:coalescence} 
Assume $x>0$ and $0 <\beta_1 , \beta_2 <1$ with $\beta_1>\frac{\beta_2}{1+2\beta_2}$.
Then the hitting time $T^\star=\inf \{ t \ge 0 \mid   X_t^{x,\beta_1}=X_t^{x,\beta_2} \}$ is almost surely finite.
\end{theo}
\begin{rem} Actually in \cite{BarBur_etal01} the case $\beta_1=\beta_2$ is considered with $x>0$, and in
\cite{BurChe01} the situation $\beta_1 \neq \beta_2$ is treated in the case $x=0$ and with the condition
$ \frac{\beta_2}{1+2\beta_2}<\beta_1<\beta_2$. Nevertheless, it is rather clear that the additional condition $\beta_1<\beta_2$
is mainly related to the choice $x=0$ and could be removed if $x>0$. 
However, we will give below a new proof 
of Theorem \ref{T:coalescence}. 
\end{rem}
In \cite{BurChe01} the law of  $U^\star=L^0_{T^\star}(X^{0,\beta_1})$ is computed in the particular situation $\beta_1=\beta_2$. In the following theorem we compute the law without this restriction.

\begin{theo} \label{T:law_U}
Assume $x>0$ and $0 <\beta_1 , \beta_2 <1$ with $\beta_1>\frac{\beta_2}{1+2\beta_2}$. 
Denote $U^\star=L^0_{T^\star}(X^{0,\beta_1})$ then the law of $U^\star$ 
has the density
\begin{equation} \label{E:law_U_star}
p_{U^\star}(x, \dd u) = 
\frac{1}{\textbf{b}(1-\xi^\star,\frac{1-\beta_1}{2\beta_1})} 
\frac{\beta_1}{x}
\left( \frac{\beta_1 u}{x} \right)^{\xi^\star -2}
\left(1-\frac{x}{\beta_1 u}\right)^{\frac{1-3\beta_1}{2\beta_1}}
\ind_{[\frac{x}{\beta_1},\infty)}(u)
\dd u 
\end{equation}
where $\ds \textbf{b}(a,b)=\int_0^1 u^{a-1} (1-u)^{b-1} \dd u=\frac{\Gamma(a)\Gamma(b)}{\Gamma(a+b)}$ and 
$\xi^\star=\frac{1}{2\beta_1}-\frac{1}{2\beta_2}$.

Hence, $\frac{x}{\beta_1 U^\star}$ is distributed as a Beta random variable $\mathcal{B}(
1-\xi^\star,\frac{1-\beta_1}{2\beta_1})$.
\end{theo}
\begin{rem} For $\beta_1=\beta_2$ we retrieve the result of \cite{BurChe01}. However,
in \cite{BurChe01} the cumulative distribution function of $U^\star$ was explicitly
derived using a max-stability argument for the law of $U^\star$. By \eqref{E:law_U_star} 
we see that for $\beta_1 \neq \beta_2$ the cumulative distribution function cannot be computed explicitly. Actually,
arguments similar to \cite{BurChe01} 
do not seem to apply directly here.
\end{rem}


The following proposition deals with the finiteness of the hitting time of $X^{0,\beta_1}$ and $X^{x,\beta_2}$ when one of the skewness parameters is negative. It can be easily derived from
Theorem \ref{T:coalescence}; a proof is given in Section \ref{S:negative_beta}.
\begin{prop} \label{P:coalescence_negative}
Assume $x>0$ and $-1<\beta_2<0<\beta_1<1$, then $T^\star$ is almost surely finite.
\\
Assume $x>0$ and $-1<\beta_1<0<\beta_2<1$, then $T^\star=\infty$ almost surely.
\end{prop}
We can compute the law of the hitting time when the skewness parameters have different signs.
\begin{theo} \label{T:law_U_diff_sign}
Assume $x>0$ and $-1<\beta_2<0<\beta_1<1$,
then the law of $U^\star=L^0_{T^\star}(X^{0,\beta_1})$ has the density
\begin{equation} \label{E:law_U_star_beta_neg}
p_{U^\star}(x, \dd u) = 
\frac{1}{\textbf{b}(\frac{\beta_2-1}{2\beta_2},\frac{1-\beta_1}{2\beta_1})}
\frac{\beta_1}{x}
\left( \frac{\beta_1 u}{x}\right)^{-\frac{1+\beta_2}{2\beta_2}}
\left( 1-\frac{\beta_1 u}{x} \right)^{\frac{1-3\beta_1}{2\beta_1}}
\ind_{[0,\frac{x}{\beta_1}]} (u) \dd u.
\end{equation}
Hence, $\frac{\beta_1 U^\star}{x}$ is distributed as a Beta random variable $\mathcal{B}(\frac{\beta_2-1}{2\beta_2},\frac{1-\beta_1}{2\beta_1})$.
\end{theo} 

It remains to study the case where $\beta_1<0$, $\beta_2<0$.
We have the following result, which will be deduced from the previous ones.
\begin{coro} \label{C:x_neg}
Assume $x>0$ with $-1<\beta_1, \beta_2<0$ and $\abs{\beta_2} > \frac{\abs{\beta_1}}{1+2 \abs{\beta_1}}$. Then $T^\star$ is finite and $\left( 1-\frac{\beta_1 L^{0}_{T^\star}(X^{0,\beta_1})}{x}\right)^{-1}$ is distributed as a product of two independent Beta variables.
\end{coro}
\begin{rem} Remark that the condition $x>0$ in the previous results is essentially irrelevant. Indeed if $x<0$, we may set $\widetilde{X}^x=-X^x$, $\widetilde{X}^0=-X^0$, $\widetilde{\beta_1}=-\beta_1$ and
$\widetilde{\beta_2}=-\beta_2$. This simple transformation reduces the situation to one of those studied in Theorems \ref{T:law_U}--\ref{T:law_U_diff_sign} or Corollary \ref{C:x_neg}.
\end{rem}


Throughout all the paper, the parameter
$\beta_1$ is associated to the process starting from $0$ and $\beta_2$ to the one starting from $x>0$, so we will,
from now on, suppress the dependence upon the skewness parameters and write $X^0$, $X^x$, $Z^x$ for $X^{0,\beta_1}$, $X^{x,\beta_2}$, $Z^{x,\beta_1,\beta_2}$.  
Moreover, we shall only consider the inverse of local time for the process 
$X^{0}$ 
and hence we shall write $\tau_u$ for $\tau_u(X^0)$ when no confusion is possible. 

Let us introduce, for $u\ge 0$, the sigma field,
\begin{equation}
\label{E:def_G}
\mathcal{G}_u=\mathcal{F}_{\tau_u}.
\end{equation}
With these notations, the process $(Z^x_u)_{u \ge 0}$ is $(\mathcal{G}_u)_{u \ge 0}$ adapted.
Moreover we can see that its law defines a Markov semi group.
Indeed, we can use the a.s. relation $\tau_l(X^{0}_{\tau_h+\cdot})=\tau_{h+l}(X^{0})-\tau_{h}(X^0)$ to get
\begin{equation}\label{E:proof_Z_makov}
Z^x_{h+l}=X^{x}_{\tau_{h+l}(X^0)}=X^x_{\tau_h(X^0)+\tau_l(X^{0}_{\tau_h+\cdot})}.
\end{equation}
Then, using the pathwise uniqueness for the skew equations, 
 we see that the law of $(X^x_{\tau_h+\cdot},X^0_{\tau_h+\cdot})$
conditional to $\mathcal{F}_{\tau_h}$ is the law of solutions to \eqref{E:defXx}--\eqref{E:defX0}
starting from $(X^x_{\tau_h},X^0_{\tau_h})=(X^x_{\tau_h},0)$. This fact with \eqref{E:proof_Z_makov} shows that the law of $(Z_u^x)_{u \ge 0}$ defines a Markov semi group.

Consequently, we remark that $U^\star= \inf \{ u \ge 0 \mid Z^{x}_u=0 \}$ is the a hitting time of a Markov process.
This is the crucial fact that allows us to compute the law of $U^\star$.

In the next Section, we will study the dynamics of the Markov process $Z^x$ and, in particular, give the proof
of Theorem \ref{T:SDEJump_main}.
For simplicity, we have decided to focus the paper mainly on the situation $0<\beta_1,\beta_2<1$.
This restriction especially holds true in the Sections \ref{S:SDE_with_jumps}--\ref{S:Hitting} below.

\section{Stochastic differential equation with jumps characterisation of $Z^x$ (case $0<\beta_1,\beta_2<1$)}
\label{S:SDE_with_jumps}

In this section we assume that we are in the situation  $0<\beta_1,\beta_2<1$.
We will show that $Z^x$ is solution to some stochastic differential equation governed by
the excursion point Poisson process of $X^0$.


First, we recall 
some basic facts about the excursion theory.

\subsection{Excursions of a skew Brownian motion}
Consider $X^{0,\beta}$ a skew Brownian motion starting from $0$ and introduce the inverse of its local time
$\tau_u(X^{0,\beta}) = \inf \{ t \ge 0 \mid L_t^{0} (X^{0,\beta}) > u \}$.
Recall that the excursion process $(\textbf{e}_u)_{u >0}$ associated to $X^{0,\beta}$ is
$\textbf{e}_u(r)=X^{0,\beta}_{\tau_{u-}(X^{0,\beta}) + r},$
for $r \le \tau_{u}(X^{0,\beta}) - \tau_{u-}(X^{0,\beta})$.
The Poisson point process $(\textbf{e}_u)_{u >0}$ takes values in the space $\exspace$
of 
excursions.
For $\textbf{e} \in \exspace$ we denote $R(\textbf{e})$ the lifetime of the excursion and recall that by definition $\textbf{e}$ does not hit zero on $(0,R(\textbf{e}) )$, and $\textbf{e}(r)=0$ for $r \ge   R(\textbf{e})$.

If we denote $\textbf{n}_\beta$ the excursion measure of the $X^{0,\beta}$, we have the formula, for 
$A$ any Borel subset of $\exspace$,
\begin{equation} \label{E:n_nBM}
\textbf{n}_\beta(A)= \frac{(1+\beta)}{2}\textbf{n}_{\abs{\text{B.M.}}}(A) + \frac{(1-\beta)}{2} 
\textbf{n}_{\abs{\text{B.M.}}}(-A)
\end{equation}
where $\textbf{n}_{\abs{\text{B.M.}}}$ is the excursion measure for 
the absolute value of a Brownian motion. 
Let us recall some useful facts on the excursion measure $\textbf{n}_{\abs{\text{B.M.}}}$, that are
immediate from well known properties of the excursion measure of a standard Brownian motion. 

First, we recall the law of the height of an excursion (for example see chapter 12 in \cite{RevYor}):
\begin{equation}\label{E:n_reaches}
\textbf{n}_{\abs{\text{B.M.}}}( \textbf{e} \text{ reaches } h) =\frac{1}{h}, \quad \text{ for } h>0.
\end{equation}


Second, we recall that in the case of a standard Brownian motion, the law of the excursion after reaching some fixed level $h$, is the same as the law of a Brownian motion starting from $h$ before it hits $0$. 
We rewrite precisely this property for the reflected Brownian motion as follows.
Let $G: 
\mathcal{C}( [0,\infty) ,\mathbb{R})  \to \mathbb{R}_+$ be some measurable functional on the canonical Wiener space.
 For $h \in \mathbb{R}$ denote
$T^h(\textbf{e})=\inf\{s \mid \textbf{e}_s=h \}$ and let $w^h_r:=w^h_r(\textbf{e}):=\textbf{e}_{T^h(\textbf{e})+r}-h$  for 
$r \le R(\textbf{e})-T^{h}(\textbf{e})=T^{-h}(w^h(\textbf{e}))$ be the shifted part of the excursion after $T^h$.
Then by Theorem 3.5 p. 491 in \cite{RevYor}, for $h>0$,
\begin{align*}
\frac{ 
\textbf{n}_{\abs{\text{B.M.}}} \left[
G( \textbf{e}_{ (T^h(\textbf{e})+ \cdot ) \wedge R(\textbf{e})} -h ) \ind_{ \{ \text{ \textbf{e} reaches } h \}}
 \right]
}{\textbf{n}_{\abs{\text{B.M.}}} \left[ 
\text{ \textbf{e} reaches } h
\right]
}
&=
\frac{ 
\textbf{n}_{\abs{\text{B.M.}}} \left[
G( w^h_{\cdot \wedge T^{-h} (w^h)} ) \ind_{ \{ \text{ \textbf{e} reaches } h \}}
 \right]
}{\textbf{n}_{\abs{\text{B.M.}}} \left[ 
\text{ \textbf{e} reaches } h
\right]
}
\\
&= \int_{\mathcal{C}([0,\infty),\mathbb{R})}  G (w_{. \wedge  T^{-h}(w)}) \dd \mathbb{W}(w)
\end{align*} 
where $\mathbb{W}$ is the standard Wiener measure.

Using \eqref{E:n_nBM} we deduce that for $h \neq 0$,
\begin{equation} \label{E:fundamental_pte_n}
\frac{  \textbf{n}_\beta \left[
G( w^h_{\cdot \wedge T^{-h} (w^h)} ) \ind_{ \{  \textbf{e} \text{ reaches } h \}}
 \right] }
{\textbf{n}_\beta \left[ 
\text{ \textbf{e} reaches } h
\right]
}
= \int_{\mathcal{C}([0,\infty),\mathbb{R})}  G (w_{. \wedge  T^{-h}(w)}) \dd \mathbb{W}(w).
\end{equation}

\subsection{Representation of the local time of $X^x$ as a functional of the excursion process of $X^0$.}
The next Proposition shows that $L^0_{\tau_u}(X^x)$ is a functional of $(\textbf{e}_u)_{u>0}$ the excursion process of $X^0$ (recall that $\tau_u=\tau_u(X^{0,\beta_1})$). 
\begin{prop}  \label{P:L0_somme_ell}
Almost surely, one has the representation for
all $t < U^\star$ 
\begin{equation} \label{E:L0_somme_ell}
L^0_{\tau_t}(X^x) = \sum_{0 < u \le t} \ell ( X_{\tau_{u-}}^x, {\rm{\bf e}}_u), 
\end{equation}
where $\ell: (0,\infty) \times \exspace \to [0,\infty)$ is a measurable map. 

For $h>0$, we can describe the law of ${\rm{\bf e}} \mapsto \ell(h,{\rm{\bf e}})$ under ${\rm{\bf n}}_{\beta_1}$ by
\begin{equation} \label{E:law_ell}
{\rm{\bf n}}_{\beta_1} ( \ell(h,{\rm{\bf e}}) \ge a) = 
\frac{1-\beta_1}{2h}
\left( 1+ \frac{\beta_2 a}{h} \right)^{-\frac{1+\beta_2}{2\beta_2}}, \quad 
\forall a >0.
\end{equation}  
\end{prop}
\begin{proof}
Before turning to a rigorous proof, let us give some insight about the representation \eqref{E:L0_somme_ell}.
Assume $\tau_u-\tau_{u-}>0$, then using \eqref{E:defX0},  we have for $r \le R(\textbf{e}_u)$,
\begin{align*}
\textbf{e}_u(r)=X^0_{\tau_{u-}+r}
&=
 X^0_{\tau_{u-}}+ B_{\tau_{u-}+r}-B_{\tau_{u-}}+\beta_1
[ L^{0}_{\tau_{u-}+r}(X^0)-L^{0}_{\tau_{u-}}(X^0)]
\\
&= B_{\tau_{u-}+r}-B_{\tau_{u-}}.
\end{align*}
Recalling \eqref{E:defXx}, we deduce,
\begin{align}
\nonumber
X^x_{\tau_{u-}+r}
&=
 X^x_{\tau_{u-}}+ B_{\tau_{u-}+r}-B_{\tau_{u-}}+\beta_2
[ L^{0}_{\tau_{u-}+r}(X^x)-L^{0}_{\tau_{u-}}(X^x) ]
\\ \nonumber
&= 
 X^x_{\tau_{u-}}+ \textbf{e}_u(r)
 +\beta_2[L^{0}_{\tau_{u-}+r}(X^x)-L^{0}_{\tau_{u-}}(X^x)]
 \\ &=
 \label{E:eq_sBm_e_u}
  X^x_{\tau_{u-}}+ 
  \textbf{e}_u(r)
 + \beta_2 L^{0}_{r}(X^x_{\tau_{u-}+\cdot}).
 \end{align}
The relation \eqref{E:eq_sBm_e_u} shows that $(X^x_{\tau_{u-}+r})_{r<R(\textbf{e}_u)}$ satisfies a skew Brownian motion type of equation, but
 governed by the excursion path $\textbf{e}_u$, and starting from the value $X^x_{\tau_{u-}}$. 
By solving this equation, we will show that the process $(X^x_{\tau_{u-}+r})_{r<R(\textbf{e}_u)}$ can be obtained as a functional of the excursion 
$\textbf{e}_u$ and of the initial value $X^x_{\tau_{u-}}$.
 As a consequence the local time $L^{0}_{\tau_{u}}(X^x)-L^{0}_{\tau_{u-}}(X^x)$ will be written too as
a functional $\ell(X_{\tau_{u-}^x},\textbf{e}_u)$.
We give a rigorous proof of these facts in the first two steps below. 
Remark that in a general way,
it is not true that the equation $\widehat{X}_r = h + \textbf{e}(r)+ \beta_2 L^0_r(\widehat{X})$ admits a unique
solution for all $h \in \mathbb{R}$ and all $\textbf{e} \in \exspace$. This makes the rigorous construction a bit delicate.  

%
%
%

{\bf Step 1: Construction of solutions to the skew equation driven by an excursion.}

Consider 
$\mathcal{C}([0,\infty),\mathbb{R})$ the canonical space endowed with 
$\mathbb{W}$ the measure of the standard Brownian motion (starting at zero).
Since the skew Brownian motion equation admits a unique strong solution,
we know that there exists a solution 
$(\mathcal{X}_r)_{r \ge 0}$ to the equation
\begin{equation} \label{E:skew_canonic}
\mathcal{X}_r(\omega)=
\omega_r + \beta_2 L^0_r(\mathcal{X}(\omega))
\end{equation}
as long as $\omega \in \widehat{\Omega}$ where $\widehat{\Omega}$ is some subset of
$\mathcal{C}([0,\infty),\mathbb{R})$
with $\mathbb{W}(\widehat{\Omega} )=1$.

For $h>0$, define $T^h(\omega)=\inf \{ u>0 \mid w_u =h\}$, one can easily see that,
for any $h>0$, the process 
$\mathcal{X}_{\cdot \wedge T^h(\omega) }(\omega)$ is some functional 
of $\omega_{\cdot \wedge T^h(\omega)}$
(it can be seen, for example, and up to restricting $\widehat{\Omega}$, using that \eqref{E:skew_canonic} is a limit of S.D.E. with smooth coefficients).
 With slight abuse of notation, we write
$\mathcal{X}_{\cdot \wedge T^h(\omega) }(\omega)=\mathcal{X}_{\cdot \wedge T^h(\omega) }(\omega_{\cdot \wedge T^h(\omega)} )$.
Define $\widehat{\Omega}^h= \{ (\omega_{r \wedge T^h(\omega)})_{r \ge 0} 
\mid \omega \in \widehat{\Omega} \}$, by construction,
\begin{equation}\label{E:W_omega_h}
\mathbb{W}\left( \left\{ \omega \in \mathcal{C}([0,\infty),\mathbb{R}) \mid \omega_{\cdot \wedge T^h(\omega)} \in \widehat{\Omega}^h 
 \right\} \right)
=\mathbb{W}( \widehat{\Omega})=1.
\end{equation}
   
Now, we construct a solution of the skew equation driven by the 'generic' excursion $\textbf{e}$ and starting from an
arbitrary value $h>0$ as follows.

%

For $h>0$ and $\textbf{e} \in \exspace$:
\begin{itemize}
\item  If $\textbf{e}$ does not reach $-h$ we simply
set 
\begin{equation} \label{E:def_X_hat1}
\widehat{X}_r(h,\textbf{e})=h + \mathbf{e}(r) \text{ for $r \in [0,R(\textbf{e})]$.}
\end{equation}

\item
If $\textbf{e}$ reaches $-h$, denote $T^{-h}(\textbf{e})=\inf\{r \mid \textbf{e}(r)=-h  \}$ and
$\omega^{-h}_{\cdot}=\textbf{e}(T^{-h}(\textbf{e})+\cdot)+h$.
We have $R(\textbf{e})-T^{-h}(\textbf{e})=T^{h}(\omega^{-h})$, and if
$w^{-h} \in \widehat{\Omega}^{h}$
we set,
\begin{equation} \label{E:def_X_hat2}
\widehat{X}_r(h,\mathbf{e})=
\begin{cases}
h + \textbf{e}(r), \quad \text{ for $r \le T^{-h}(\textbf{e})$}
\\
\mathcal{X}_{r-T^{-h}(\textbf{e})}(\omega^{-h}_{\cdot \wedge T^{h}(\omega^{-h})} ) 
\text{ for $ r \in (T^{-h}(\textbf{e}) ,R(\textbf{e})]
$}
\end{cases}.
\end{equation}
\item
If $\textbf{e}$ reaches $-h$ and $w^{-h} \notin \widehat{\Omega}^{h}$ we arbitrarily set 
\begin{equation} \label{E:def_X_hat3}
\widehat{X}_r( \textbf{e} ,h)=h \text{ for all $r \in [0,R(\textbf{e}))$.}
\end{equation}
\end{itemize}

Remark that \eqref{E:def_X_hat2} simply means that after the excursion reaches $-h$, we use the solution of the skew equation defined on the canonical space, treating the part of the excursion after $T^{-h}(\textbf{e})$ as the realisation of the Brownian motion when such construction is feasible.

Since for $\omega \in \widehat{\Omega}$, $\mathcal{X}(\omega)$ satisfies \eqref{E:skew_canonic}, and the local time of $\widehat{X}(\mathbf{e},h)$ does not increase before $T^{-h}(\textbf{e})$, we have 
$\forall \textbf{e} \in \exspace$, such that  
$\textbf{e}$ reaches
$-h$ with $\omega^{-h} \in \widehat{\Omega}^h$:
\begin{equation} \label{E:equation_X_hat}
\widehat{X}_r(h,\textbf{e})= h + \textbf{e}(r)+
\beta_2 L^{0}_r(\widehat{X}(h,\textbf{e})), \quad
\forall r \le R(\textbf{e}).
\end{equation}
Observe that if $\textbf{e}$ does not reach $-h$, it is immediate that the relation \eqref{E:equation_X_hat} is satisfied.

Hence, we can deduce that 
for $h>0$ fixed, the $\textbf{n}_{\beta_1}$ measure of the set of excursions $\textbf{e}$ where \eqref{E:equation_X_hat} does not hold is zero: indeed, using the fundamental property \eqref{E:fundamental_pte_n} of the excursion measure,
\begin{align}  \label{E:n_exceptionel_first}
\textbf{n}_{\beta_1}
\left( 
\textbf{e} (T^{-h}(\textbf{e})+\cdot)+h \notin \widehat{\Omega}^{h}
 \mid \textbf{e} \text{ reaches $-h$ } 
\right)&=
\textbf{n}_{\beta_1}\left( \omega^{-h}_{\cdot \wedge T^{h}(\omega^{-h})} 
\notin \widehat{\Omega}^{h} \mid \textbf{e} \text{ reaches $-h$ } 
\right)
\\ \label{E:n_exceptionel}
&=\mathbb{W}\left( \left\{ \omega \mid \omega_{\cdot \wedge T^h(\omega)} \notin \widehat{\Omega}^h 
 \right\} \right)
=0, 
\end{align}
where we have used \eqref{E:W_omega_h}.

We finally define the total local time of $\widehat{X}(h,\textbf{e})$ during the lifetime of the excursion by setting:
\begin{equation} \label{E:def_ell}
\ell(h, \textbf{e})=
\begin{cases}
0, ~ \text{ if $\textbf{e}$ does not reach $-h$,}
\\
L^0_{T^{h}(\omega^{-h})}(\mathcal{X}(\omega^{-h}_{\cdot \wedge T^{h}(\omega^{-h}) } ))=
L^0_{R(\textbf{e})}(\widehat{X}(h,\textbf{e})), ~
\text{ if  $\textbf{e}$ reaches $-h$ and $\omega^{-h} \in \widehat{\Omega}^{h}$,}
\\
0, ~  \text{ if  $\textbf{e}$ reaches $-h$ and $\omega^{-h} \notin \widehat{\Omega}^{h}$.}
\end{cases} 
\end{equation}

{\bf Step 2: Proof of the relation \eqref{E:L0_somme_ell}.}

For $s < T^\star$, we have $X^x_s > X^0_s$ and we deduce that the local time of $X^x$ does not increase
on the set $\{ s < T^\star \mid X^0_s=0 \}=
[0,T^\star)\setminus \cup_{u<U^\star} ]\tau_{u-},\tau_u[  $.
Hence, we have for $t < U^\star$ the relation
\begin{equation*}
L^{0}_{\tau_t}(X^x)= \sum_{0<u \le t}  [
L^{0}_{\tau_{u}}(X^x)-L^{0}_{\tau_{u-}}(X^x) ].
\end{equation*}
Now it is clear that \eqref{E:L0_somme_ell} will be proved if we show that almost surely,
\begin{equation} \label{E:delta_L_ell}
L^{0}_{\tau_{u}}(X^x)-L^{0}_{\tau_{u-}}(X^x)
=\ell(X^x_{\tau_{u-}}, \textbf{e}_u ), \quad 
\text{ for all $u$ with $\tau_{u}-\tau_{u-}>0$}.
\end{equation}

Actually, we will construct $\widetilde{X}^x$, a version of $X^x$, which satisfies the relation \eqref{E:delta_L_ell}. 
Let $U_{1}= \inf \{ u >0 \mid \textbf{e}_u \text{ reaches } -x+\beta_1  u \}$. Then $U_{1}$ is a 
$(\mathcal{G}_u)_{u \ge 0}$ stopping time and it is immediate that $U_1 <x/\beta_1$ almost surely.
We construct $\widetilde{X}^x$ on $[0,T_1]$ where $T_1=\tau_{U_1}$ in the following way.

For $t<\tau_{U_1 -}$ we set 
\begin{equation*}
\widetilde{X}^x_t=
\begin{cases}
x+e_s(t-\tau_{s-})-\beta_1 s, \text{ if $t \in (\tau_{s-},\tau_{s})$ with $s<U_1$,}
\\
x-\beta_1 L_t^0(X^0), \text{ if $t \in [0,\tau_{U_1-}) \setminus \cup_{s < U_1}
(\tau_{s-},\tau_s) = \{ v< \tau_{U_1-} \mid X_v^0=0 \}$,}
\end{cases}
\end{equation*}
and for $t \in [\tau_{U_1 -} ,\tau_{U_1}]$ we let
\begin{equation} \label{E:tildeX_saut}
\widetilde{X}^x_t= \widehat{X}_{t-\tau_{U_1 -}}(\widetilde{X}^x_{\tau_{U_1 -}}, \textbf{e}_{U_1}  )
=\widehat{X}_{t-\tau_{U_1 -}}(x-\beta_1 U_1, \textbf{e}_{U_1}  ),
\end{equation}
where $\widehat{X}$ was defined in \eqref{E:def_X_hat1}--\eqref{E:def_X_hat3}.
Note that, with this definition,
\begin{equation}
\label{E:inc_TL_Xhat}
L^0_t(\widetilde{X}^x)-L^0_{\tau_{U_1-}}(\widetilde{X}^x)
=
L^0_{t-\tau_{U_1-}}(\widehat{X}(x-\beta_1 U_1,\textbf{e}_{U_1})), 
\text{ for $t \in [\tau_{U_1-},\tau_{U_1}]$.}
\end{equation}

Let us check that $\widetilde{X}^x$ satisfies the skew equation.
By definition of $U_1$ we have $\widetilde{X}^x_t>0$ on $[0, \tau_{U_1-})$. Moreover,
$\textbf{e}_s(t-\tau_s)=X^0_t$ when $t \in (\tau_{s-},\tau_{s})$, together with \eqref{E:defX0} 
insures that $\widetilde{X}^x_t=x+B_t $ for $t \in [0, \tau_{U_1-})$. As a result, $\widetilde{X}^x$
satisfies the skew equation on $[0, \tau_{U_1-})$.

We now focus on the interval $[\tau_{U_1-},\tau_{U_1})$. First, using 
the so-called ``Master Formula'' (Proposition 1.10 page 475 in \cite{RevYor}) we get
\begin{multline*}
\mathbb{E} \left[
\sum_{0<s<\frac{x}{\beta_1}} 
1_{ \{ \textbf{e} \text{ reaches $-x+\beta_1 s$} \}}
1_{\{ \textbf{e}(T^{-x+\beta_1 s}(\textbf{e}) + \cdot  ) +  x-\beta_1 s \notin \widehat{\Omega}^{x-\beta_1 s} \}}
\right]
\\ =
\int_0^{\frac{x}{\beta_1}} 
\textbf{n}_{\beta_1} ( \textbf{e} \text{ reaches $-x+\beta_1 s$ };\quad
\textbf{e}(T ^{-x+\beta_1 s}(\textbf{e}) + \cdot  ) +  x-\beta_1 s \notin 
\widehat{\Omega}^{x-\beta_1 s}\}
 ) \dd s
\end{multline*} 
and the latter integral is zero from \eqref{E:n_exceptionel_first}--\eqref{E:n_exceptionel}.
Hence, we deduce that with probability one, for all $u < \frac{x}{\beta_1}$ the relation \eqref{E:equation_X_hat} holds true with $\mathbf{e}$ replaced by $\mathbf{e}_u$ and $h$ replaced by $x-\beta_1 u$. 
As a consequence, it holds true, almost surely, for the excursion occurring at the random time $U_1$.
This yields to the following relation
for $t \in [\tau_{U_1-},\tau_{U_1}]$,
\begin{align*}
\widetilde{X}^x_t&=\widetilde{X}^x_{\tau_{U_1-}} + \textbf{e}_{U_1}(t-\tau_{U_1-})+
\beta_2 L^0_{t-\tau_{U_1-}}( \widehat{X}(x-\beta_1 U_1, \textbf{e}_{U_1}))
\\
&=\widetilde{X}^x_{\tau_{U_1-}} + \textbf{e}_{U_1}(t-\tau_{U_1-})+
\beta_2 [L^0_{t}( \widetilde{X}^x) - L^0_{\tau_{U_1-}}( \widetilde{X}^x)],
\quad \text{ by \eqref{E:inc_TL_Xhat}}
\\
&=\widetilde{X}^x_{\tau_{U_1-}} + X^0_t+
\beta_2 [L^0_{t}( \widetilde{X}^x) - L^0_{\tau_{U_1-}}( \widetilde{X}^x)],
\quad \text{ by definition of the excursion process}
\\
&=\widetilde{X}^x_{\tau_{U_1-}} + B_t + \beta_1 L^0_t(X^0)+
\beta_2 [L^0_{t}( \widetilde{X}^x) - L^0_{\tau_{U_1-}}( \widetilde{X}^x)],
\quad \text{ by \eqref{E:defX0}}
\\
&=x +B_t + \beta_2 L^0_{t}( \widetilde{X}^x),
\end{align*}
where we have used $\widetilde{X}^x_{\tau_{U_1-}}=x-\beta_1 U_1$, 
$L^0_t(X^0)=U_1$, and $L^0_{\tau_{U_1-}}( \widetilde{X}^x)=0$ in the last line.
This completes the proof that $\widetilde{X}^x$ almost surely satisfies the 
 skew equation on $[0, T_1]$.
 Using pathwise uniqueness for the skew equation, we deduce that $\widetilde{X}^x=X^x$ on $[0,T_1]$, almost surely.

Then, by the definition of the functional $\ell$ in the first step of the proof, and by the construction of $\widetilde{X}$, we see that the condition \eqref{E:delta_L_ell} holds true for $u \le U^1$. We deduce
 $L^0_{\tau_t}(X^x)=\sum_{ 0<u \le t } \ell(X^x_{\tau_{u-}},\textbf{e}_u)$
for $t \le U_1$. 
Remark that, since $t \le U_1$, the only possible non zero term in this sum is $\ell(X^x_{\tau_{U_1-}},\textbf{e}_{U_1})$.

The process $\widetilde{X}^x$ is then constructed recursively.
For $i\ge 1$, we let $U_{i+1}=\inf\{ u > U_i \mid
\textbf{e}_u \text{ reaches } -\widetilde{X}^x_{T_i} + \beta_1 (u - U_i) \}$, we set 
$T_{i+1}=\tau_{U_{i+1}}$. We define $\widetilde{X}^x$ on $[T_i,T_{i+1}]$ as follows,
\begin{align*}
&\text{ if $t \in [T_i, \tau_{U_{i+1}-})$ with $t \in (\tau_{s-},\tau_{s})$ for some $s$:} \quad 
\widetilde{X}^x_t=\widetilde{X}^x_{T_i}+\textbf{e}_s(t-\tau_{s-})-\beta_1(s-U_i)\\
&\text{ if $t \in [T_i, \tau_{U_{i+1}-})$ with $X^0_t=0$:} \quad 
\widetilde{X}^x_t=\widetilde{X}^x_{T_i}-\beta_1(L_t^0(X^0)-U_i)\\
&\text{ if $t \in [\tau_{U_{i+1}-},T_{i+1}]$:} \quad 
\widetilde{X}^x_t
\begin{array}[t]{l}
=\widehat{X}^x_{t-\tau_{U_{i+1}-}} (\widetilde{X}_{\tau_{U_{i+1}-}},\textbf{e}_{U_{i+1}})\\
= \widehat{X}^x_{t-\tau_{U_{i+1}-}} (\widetilde{X}_{T_i}-\beta_1(U_{i+1}-U_{i}),\textbf{e}_{U_{i+1}}).
\end{array}
\end{align*}
Then, with arguments similar to the one used on $[0,T_1]$, 
we can prove
that $\widetilde{X}^x$ satisfies the skew equation on $[T_{i},T_{i+1}]$ and consequently, if $X^x$ and $\widetilde{X}^{x}$ coincide at the instant $T_i$, they must coincide almost surely on $[T_i,T_{i+1}]$.

Using a recursion argument, we construct a process $\widetilde{X}^x$ on 
$[0, \sup_i T_i)$, which is a.s. equal to $X^x$. Moreover by construction the relation
\eqref{E:delta_L_ell} is valid for $u < \sup_{i} U_i$. 
To get \eqref{E:L0_somme_ell} we need to 
check that, almost surely,  $\sup_{i} U_i=U^\star$ or equivalently that,
$\sup_i T_i=T^\star=\inf \{ t>0 \mid X^x_t=X^0_t\}$.  

This is immediate if $\sup_{i \ge 1} T_i=\infty$. Assume, by contradiction that the set
$\{\sup_{i \ge 1} T_i<\infty; \sup_{i \ge 1} T_i<T^\star \}$ does not have
probability zero.
On this set, we have $\widetilde{X}^x_t-X^0_t=X^x_t-X^0_t \ge \varepsilon>0$ for some random $\varepsilon$
and $t$ belonging to some random left-neighbourhood of $\sup_{i \ge 1} T_i$. But, it can be seen, from
the definition of the jump times $U_i$, that there is only a finite number of jumps when
 $\widetilde{X}^x_t-X^0_t$ remains above the level $\varepsilon$ (see too Remark \ref{R:activ_finite} below).
This is in contradiction with the existence of the accumulation point $\sup_i {T_i}$, and as a result we deduce that $\sup_i T_i=T^\star$.

{\bf Step 3: Law of $\textbf{e} \mapsto \ell(h,\textbf{e})$ under the excursion measure 
($h>0$ fixed).}

Let $a>0$, using the fundamental property \eqref{E:fundamental_pte_n} of the excursion measure
and the definitions \eqref{E:skew_canonic} \eqref{E:def_ell},
 we have
\begin{equation} \label{E:eq1_step3}
\frac{\textbf{n}_{\beta_1}\left( \ell(h,\textbf{e}) > a; ~ \text{ \textbf{e} reaches $-h$} \right)}
{\textbf{n}_{\beta_1}\left( \text{ \textbf{e} reaches $-h$} \right)}
=
\mathbb{W}\left( L^0_{T^h(\omega)}(\mathcal{X}(\omega)) > a \right),
\end{equation}
where $\mathcal{X}$ solves $\mathcal{X}_r=\omega_r+\beta_2 L_r^0(\mathcal{X})$ and
$(\omega_r)_{r\ge 0}$ is a standard Brownian motion under $\mathbb{W}$.
We can compute
\begin{align*}
\mathbb{W}\left( L^0_{T^h(\omega)}(\mathcal{X}(\omega)) > a \right)
&=
\mathbb{W}\left( \text{no excursion of $\mathcal{X}_\cdot$ crossed over 
$h +\beta_2 L^0_\cdot(\mathcal{X})$ before $L^0_\cdot(\mathcal{X})$ reaches $a$} \right)
\\
&=
\exp \left( - \int_0^a \textbf{n}_{\beta_2} \left[ \textbf{e} \text{ reaches level } (h+\beta_2 u) \right] \dd u 
\right)
\end{align*}
where in the last line we have used that the measure of excursion of the process $\mathcal{X}$ is 
$\textbf{n}_{\beta_2}$, together with standard computations on Poisson processes.
Recalling \eqref{E:n_nBM}--\eqref{E:n_reaches}, we have
$\textbf{n}_{\beta_2} [ \textbf{e} \text{ reaches level } (h+\beta_2 u)]
=\frac{1+\beta_2}{2(h+\beta_2 u)}$ and
 we easily get that
\begin{equation} \label{E:W_step3}
\mathbb{W}\left( L^0_{T^h(\omega)}(\mathcal{X}(\omega)) > a \right)=
\left( 1+ \frac{\beta_2 a}{h} \right)^{-\frac{1+\beta_2}{2\beta_2}}.
\end{equation}
Finally, remark that 
\begin{align*}
\textbf{n}_{\beta_1}\left( \ell(h,\textbf{e}) > a  \right)
&=\textbf{n}_{\beta_1}\left( \ell(h,\textbf{e}) > a; ~ \text{ \textbf{e} reaches $-h$} \right)
\\
&=
\textbf{n}_{\beta_1} \left( \text{ \textbf{e} reaches $-h$} \right)
\left( 1+ \frac{\beta_2 a}{h} \right)^{-\frac{1+\beta_2}{2\beta_2}}, \quad
\text{using \eqref{E:eq1_step3}--\eqref{E:W_step3},}
\\
&=
\frac{1-\beta_1}{2 h} \left( 1+ \frac{\beta_2 a}{h} \right)^{-\frac{1+\beta_2}{2\beta_2}},
\quad 
\text{by \eqref{E:n_nBM}--\eqref{E:n_reaches}.}
\end{align*}
The proof of the proposition is completed.
\end{proof}

%
%
%
%

\subsection{Representation of the ``distance process'' as a jump Markov process (proof of 
Theorem \ref{T:SDEJump_main})}
Clearly the Theorem \ref{T:SDEJump_main} is an immediate consequence of \eqref{E:law_ell} and of the
equation \eqref{E:EDS_Z_ell} in the next result.
\begin{prop}\label{P:EDS_Z}
We have for all $t < U^\star$,
\begin{align}\label{E:EDS_Z_ell}
Z_t^x &= x-\beta_1 t + \sum_{0<u\le t} \beta_2 \ell(Z_{u-}^x,{\rm{\bf e}}_u)
\\  \label{E:EDS_Z_usual}
&=x-\beta_1 t + \int_{[0,t]\times(0,\infty)}
a \mu(\dd u, \dd a) , 
\end{align}
where $\mu(\dd u,\dd a)$ is the random jumps measure of $Z^x$ on $[0,U^\star)\times(0,\infty)$.
The compensator of the measure $\mu(\dd u,\dd a)$ is $\dd u \times \nu(Z^x_{u-},\dd a)$ with
\begin{equation} \label{E:def_compensator}
\nu(h,\dd a)= \frac{\kappa}{h^2} \left(1+\frac{a}{h}\right)^{-\gamma} \ind_{\{a >0 \}} \dd a
\end{equation}
where $\kappa=\frac{(1-\beta_1)(1+\beta_2)}{4 \beta_2}$ and $\gamma=\frac{1+3\beta_2}{2\beta_2}$.
\end{prop}
\begin{proof}
Using \eqref{E:defXx} we have $Z^x_t=X^x_{\tau_t}=x+B_{\tau_t}+\beta_2 L^0_{\tau_t}(X^x)$. Now from \eqref{E:defX0}, $0=X^0_{\tau_{t}}=B_{\tau_t}+ \beta_1 t$ and we deduce
\begin{equation*}
Z^x_t=x - \beta_1 t + \beta_2 L^0_{\tau_t}(X^x).
\end{equation*}
Hence the relation \eqref{E:L0_somme_ell} yields to \eqref{E:EDS_Z_ell}. 

The representation \eqref{E:EDS_Z_usual} is just the usual way to rewrite the stochastic differential equation with jumps: we transform  \eqref{E:EDS_Z_ell} into \eqref{E:EDS_Z_usual} by defining $\mu(\dd u,\dd a)$ as 
the sum of Dirac masses 
$\displaystyle 
\sum_{u<U^\star, Z^x_{u-}\neq Z_u^x } \delta_{(u,\Delta Z_u)}$,
and 
\eqref{E:def_compensator} appears as a direct consequence of \eqref{E:law_ell}.
\end{proof}
\begin{rem} \label{R:jump_law}
From Proposition \ref{P:EDS_Z}, we deduce that the rate for the jumps of $Z^x$ is given by,
\begin{equation} \label{E:prob_jump}
\mathbb{P}(Z^x \text{ jumps  on $[t,t+\dd t]$} \mid Z^x_{t-}=h)=\frac{1-\beta_1}{2h} \dd t.
\end{equation}
Conditionaly on $Z^x_{t-}=h$, the law of the jumps is given by
\begin{equation}\label{E:dens_jump}
\frac{1+\beta_2}{2 \beta_2 h} \left( 1+\frac{a}{h} \right)^{-\gamma} 1_{(0,\infty)}(a) \dd a.
\end{equation}
Remark that the jumps intensity of the process $Z^x$ is proportional to $1/Z^{x}$.
If a jump occurs at time $t$, then the size of the jump is proportional to $Z_{t-}^x$. 
Informally, $\Delta Z^x_t  \overset{\text{law}}{=} Z^x_{t-} J$ where $J$ has the density \eqref{E:dens_jump} with $h=1$.
\end{rem}
\begin{rem} \label{R:activ_finite}
As a consequence of Remark \ref{R:jump_law}, the number of jumps on $[0,t]$ is finite for $t < U^\star$ since the jump activity is bounded when $Z^x>\varepsilon$. 
\end{rem}
From Proposition \ref{P:EDS_Z}, we can deduce that $(Z_t^x)_{t < U^\star}$ is a local submartingale (resp. supermartingale) if $\beta_2>\beta_1$ (resp. $\beta_2<\beta_1$). 
Indeed, with simple computations 
$\int_{\exspace} \beta_2 \ell(h,\textbf{e}) \dd \textbf{n}_{\beta_1}(\textbf{e}) = \int_0^\infty  a \nu(h, \dd a)=
\beta_2 \frac{1-\beta_1}{1-\beta_2}$ is independent of $h$.
Hence, we can write
\begin{align*}
Z^x_t&=x-\beta_1 t + \beta_2 \frac{1-\beta_1}{1-\beta_2} t
+
\sum_{u \le t} \beta_2 \ell(Z^x_{u-},\textbf{e}_u) - \int_0^t 
\int_{\exspace} \beta_2 \ell(Z^x_{u-},\textbf{e}) \dd \textbf{n}_{\beta_1} ( \textbf{e}) \dd u
\\
&:= x +\frac{\beta_2-\beta_1}{1-\beta_2} t + M_t,
\end{align*}
where $(M_t)_t$ is a compensated jump process, and hence a local martingale. Remark that for $\beta_1=\beta_2$ the process $Z^x$ is a local martingale.

\section{Hitting time of the two skew Brownian motions (case $0<\beta_1,\beta_2<1$, $\beta_1 > \frac{\beta_2}{1+2\beta_2}$)}\label{S:Hitting}
In this section we prove the results of Section \ref{S:Main} corresponding to
 the situation $0<\beta_1,\beta_2<1$, $\beta_1 > \frac{\beta_2}{1+2\beta_2}$.
We start by giving a new proof of Theorem \ref{T:coalescence} (\cite{BarBur_etal01, BurChe01})
 relying on the dynamic of the process $Z^x$ .
\subsection{Finiteness of the hitting time (proof of Theorem \ref{T:coalescence})}
Let us show that if $\beta_1 > \frac{\beta_2}{1+2\beta_2}$ then $U^\star < \infty$ almost surely. We apply Ito's formula to the semi-martingale $\ln (Z^x_t)$ for $t < U^\star$,
\begin{align} \nonumber
\ln (Z^x_t)&= \ln (x) + \int_0^t \frac{ \dd Z^x_u}{Z^x_{u-}}
+
\sum_{u \le t} \left\{ \ln(Z^x_{u-}+\Delta Z^x_u) - \ln(Z^x_{u-}) - \frac{\Delta Z^x_u}{ Z^x_{u-}} \right\}
\\ &= \label{E:ln_Z}
\ln(x) - \int_0^t \frac{ \beta_1 \dd u}{Z^x_{u}} 
+ \sum_{u \le t} \ln \left( 1 + \frac{ \Delta Z^x_u}{Z^x_{u-}} \right) \text{ by \eqref{E:EDS_Z_ell}.}
\end{align}
Consider the jump process $\ds \mathcal{J}_t=\sum_{u \le t} \ln \left( 1+  \frac{ \Delta Z^x_u}{Z^x_{u-}} \right)
=\sum_{u \le t} \ln \left( 1+  \frac{\beta_2 \ell(Z^x_{u-},\textbf{e}_u ) }{Z^x_{u-}} \right)$. Its compensator can be easily computed using \eqref{E:EDS_Z_ell}--\eqref{E:def_compensator}. Indeed we have,
\begin{equation*}
\int_{\exspace} \ln \left( 1+  \frac{\beta_2 \ell(h,\textbf{e} ) }{h} \right)
\dd \textbf{n}_{\beta_1}(\textbf{e})
=
\int_0^\infty \ln \left( 1+  \frac{a}{h} \right) \nu(h,\dd a)=\frac{(1-\beta_1)\beta_2}{(1+\beta_2)h},
\end{equation*}
and hence $\ds \widetilde{\mathcal{J}}_t=\mathcal{J}_t- \frac{(1-\beta_1)\beta_2}{(1+\beta_2)} \int_0^t 
\frac{\dd u}{Z^x_u}$ is a compensated jump process. 
Using \eqref{E:ln_Z}, we can write
\begin{equation} \label{E:ln_Z_J}
\ln(Z^x_t)=\ln(x)+ \theta \int_0^t \frac{\dd u}{Z^x_u} + \widetilde{\mathcal{J}}_t,
\end{equation}
with $\theta=\frac{(1-\beta_1)\beta_2}{(1+\beta_2)}-\beta_1=
\frac{\beta_2-\beta_1(1+2\beta_2)}{1+\beta_2}<0$.

The process $\widetilde{\mathcal{J}}$ is a quadratic pure jumps local martingale and its bracket is clearly given by
\begin{equation*}
[\widetilde{\mathcal{J}},\widetilde{\mathcal{J}} ]_t=
\sum_{u \le t} \ln \left( 1+  \frac{\beta_2 \ell(Z^x_{u-},\textbf{e}_u ) }{h} \right)^2.
\end{equation*}
We deduce an expression for 
$\left\langle  \widetilde{\mathcal{J}},\widetilde{\mathcal{J}}\right\rangle_t$ by computing,
with the help of \eqref{E:def_compensator},
$\int_{\exspace} \ln \left( 1+  \frac{\beta_2 \ell(h,\textbf{e} ) }{h} \right)^2
\dd \textbf{n}_{\beta_1}(\textbf{e})
=
\int_0^\infty \ln \left( 1+  \frac{a}{h} \right)^2 \nu(h,\dd a)=\frac{c}{h}$ for some
constant $c>0$.
It follows $\left\langle  \widetilde{\mathcal{J}},\widetilde{\mathcal{J}}\right\rangle_t
=c\int_0^t \frac{\dd u}{Z^x_u}$. Using \eqref{E:ln_Z_J}, we get
\begin{equation}\label{E:explo_stoch}
\frac{\ln(Z^x_t)}{\left\langle  \widetilde{\mathcal{J}},\widetilde{\mathcal{J}}\right\rangle_t}
=
\frac{\ln(x)}{\left\langle  \widetilde{\mathcal{J}},\widetilde{\mathcal{J}}\right\rangle_t}+
\frac{\theta}{c}+
\frac{\widetilde{\mathcal{J}}_t}{\left\langle  \widetilde{\mathcal{J}},\widetilde{\mathcal{J}}\right\rangle_t}.
\end{equation}

Suppose now we are on the event $\{U^\star=\infty\}$. Then, either 
$\ds \{ \int_0^\infty \frac{\dd s}{Z_s^x} < \infty \}$ or 
$\ds \{ \int_0^\infty \frac{\dd s}{Z_s^x} = \infty \}$.

On the set $\ds \{ \int_0^\infty \frac{\dd s}{Z_s^x} = \infty \}$, we have 
$\left\langle  \widetilde{\mathcal{J}},\widetilde{\mathcal{J}}\right\rangle_\infty=\infty$ and using Kronecker's lemma
 (see Lemma \ref{L:Kronecker} in the Appendix),
\begin{equation*}
\frac{\ln(Z^x_t)}{\left\langle  \widetilde{\mathcal{J}},\widetilde{\mathcal{J}}\right\rangle_t}
\xrightarrow{t \to \infty} \frac{\theta}{c}<0.
\end{equation*}
Thus, $\mathbb{P}$ a.s. on the set $\left\langle  \widetilde{\mathcal{J}},\widetilde{\mathcal{J}}\right\rangle_\infty=\infty$, there exist
$t_0$ and $\eta>0$ such that $\ln(Z_t^x) \le -\eta \int_0^t \frac{\dd u}{Z^x_u}$ for all $t\ge t_0$. In view of Lemma \ref{L:eq_explosive} in the Appendix this is not possible.

On the set $\ds \{ \int_0^\infty \frac{\dd u}{Z_u^x} < \infty \}=
\{ \left\langle  \widetilde{\mathcal{J}},\widetilde{\mathcal{J}}\right\rangle_\infty<\infty\} $, using 
Kronecker's lemma and \eqref{E:explo_stoch} we see that $\ln (Z^x_t)$ converges as $t\to \infty$. This is clearly in contradiction with the finiteness of the integral $\int_0^\infty \frac{\dd u}{Z^x_u}$.

Hence, by contradiction, we have proved that $ U^\star<\infty $ a.s. and thus the Theorem \ref{T:coalescence} is shown.

\subsection{Computation of the law of hitting time (proof of Theorem \ref{T:law_U})}
As we already said, the main tool in order to compute the law of $U^\star$ is $Z^x$.
Let us define $A$ the generator of the process $Z^x$ by
\begin{align}\label{E:exp_generateur_abstrait}
A f (h)&= -\beta_1 f'(h)+
\int_0^\infty [ f(h+a)-f(h)] \nu(h,\dd a), 
\\ \label{E:exp_generateur}
&=-\beta_1 f'(h)+
\int_0^\infty [ f(h+a)-f(h)] \frac{\kappa}{h^2} \left( 1 + \frac{a}{h} \right)^{-\gamma} \dd a,
\end{align}
for $h>0$ and $f$ an element of $\mathcal{C}^1(0,\infty)$ bounded on $[0,\infty)$. 
Using the representation \eqref{E:EDS_Z_ell}--\eqref{E:def_compensator}, it is clear that $f(Z^x_t)-
\int_0^t Af(Z^x_u) \dd u$ with $t< U^\star$ is a compensated jump process and thus a local martingale. 

 
Before turning to the heart of the proof, we need to prove several lemmas in the next Section.

\subsection{Dynkin's formula}
 
Our first lemma shows a "Dynkin's formula" that relates the generator of the process with $U^\star$, the exit time from $(0,\infty)$. 

For $\lambda>0$ we denote
$u_\lambda(x)=\mathbb{E}_x[ e^{- \lambda U^\star}]$ where the subscript $x$ emphasizes the dependence upon the starting point of the process $Z^x$.

\begin{lem}[Dynkin's formula] \label{L:Dynkin}
1) The function $x \mapsto u_{\lambda}(x)$ is $\mathcal{C}^{\infty}(0,\infty)$ and
satisfies $\lim_{x \to 0} u_{\lambda}(x)=1$, and $\abs{u_\lambda(x)} 
\le e^{-\lambda x /\beta_1}$. 
Moreover the derivatives of $u_\lambda$ decay exponentially near 
$\infty$ and satisfy $x^k u_\lambda^{(k)}(x)=O(1)$ near $0$ (for any $k\ge 0$).

2) The function $u_{\lambda}$ is solution to the integro-differential equation:
\begin{equation}\label{E:dynkin}
 A u_\lambda(x)=\lambda u_\lambda(x), \quad \text{ for all $x >0$.}
\end{equation}
\end{lem}
\begin{proof}
1) First we show that $x\mapsto u_\lambda(x)$ is a smooth function. Denote $(U_n^x)_n$ the successive jumps of $(Z^x_u)_{u < U^{\star,x}}$, where again we stress the dependence upon $x$ as we write 
$U^{\star, x}=U^\star=\sup_n{U_n^x}$.

Since $Z^x$ evolves with the constant negative drift $-\beta_1 \dd t$ and jumps with the infinitesimal probability \eqref{E:prob_jump}, we can easily compute the law of $U_1^x$:
\begin{equation}\label{E:lawU11}
\mathbb{P}(U_1^x\ge t) = \left(1-\frac{\beta_1 t}{x}\right)^{\frac{1-\beta_1}{2\beta_1}}.
\end{equation}
 As
a result the law of $U_1^x$ is equal to the law of $x U_1^1$. Moreover, the law of the jump of
$Z^x_t$ is proportional to $Z^x_{t-}$ (see remark \ref{R:activ_finite}) and this implies that the law
of $Z^x_{U_1^x}$ and $x Z^1_{U_1^1}$ are equal. Consequently, we deduce that the processes
$(Z^x_{t})_{t \le U_1^x}$ and $(xZ^1_{t/x})_{t \le x U_1^1}$ have the same law. Then, by induction, it can be seen that the two processes $(Z^x_{t})$ and $(xZ^1_{t/x})$ have the same law up to their respective $n$-th jump time. Letting $n$ tend to infinity, we deduce,
\begin{equation}
(Z^x_t)_{t < U^{\star, x}} \overset{\text{law}}{=} (xZ^1_{t/x})_{t < xU^{\star,1}}
\text{ and } U^{\star,x}=xU^{\star,1}.
\end{equation}
Now, by definition $u_\lambda(x)=\mathbb{E}\left[e^{-\lambda U^{\star,x} } \right]=
\mathbb{E}\left[e^{-\lambda x U^{\star,1} } \right]$. In turn, $x \mapsto u_\lambda(x)$ is clearly
a
$\mathcal{C}^\infty(0,\infty)$ function.
Moreover, using $U^{\star,1}<\infty$ and 
Lebesgue's theorem we get $u_{\lambda}(x) \xrightarrow{ x \to 0} \mathbb{E}[e^0]=1$.  

Next, by \eqref{E:EDS_Z_ell} and the positivity of the jumps of $Z^x$, one must have
$U^{\star,x} \ge x/ \beta_1$ a.s. and 
thus $u_\lambda(x)=\mathbb{E}\left[e^{-\lambda U^{\star,x} } \right]\le e^{-\lambda x /\beta_1}$.
In the same way, from $u^{(k)}_\lambda(x)= \mathbb{E}\left[ (-\lambda U^{\star,1})^k e^{-\lambda x U^{\star,1} } \right]$ we easily deduce the exponential decay of $u^{(k)}_\lambda$ near $\infty$. Finally, the boundedness $x^k u^{(k)}_\lambda(x)$ is clear too from the latter representation of $u^{(k)}_\lambda(x)$.

%

2) We now prove the Dynkin's equation \eqref{E:dynkin}. Consider the martingale $(M_t)_{t \ge 0}$ defined as $M_t=\mathbb{E} [e^{-\lambda U^{\star,x}} \mid \mathcal{G}_{t \wedge U^{\star,x}}]$ for $t \ge 0$ (recall the notation \eqref{E:def_G}). 
We can write,
\begin{align} \nonumber
M_t \ind_{ \{ t < U^{\star,x} \} } &= 
\mathbb{E} [e^{-\lambda U^{\star,x}} \mid \mathcal{G}_{t \wedge U^{\star,x}}] \ind_{ \{ t < U^{\star,x} \} }
=
\mathbb{E} [e^{-\lambda U^{\star,x}} \ind_{ \{ t < U^{\star,x} \} } \mid \mathcal{G}_{t \wedge U^{\star,x}}] \ind_{ \{ t < U^{\star,x} \} }
\\
\label{E:proof_Dynkin1}
&=
\mathbb{E} [e^{-\lambda U^{\star,x}} \ind_{ \{ t < U^{\star,x} \} } \mid \mathcal{G}_{t}] \ind_{ \{ t < U^{\star,x} \} }.
\end{align}
For $t>0$,  denote
 $U^{\star}(Z^x_{t+\cdot})=\inf \{ s \ge 0 \mid Z^x_{t+s}  \le 0  \}$ and remark that on the set 
 $U^{\star,x} >t$ we have $U^{\star}(Z^x_{t+\cdot})=U^{\star,x} -t$ and thus,
 $e^{-\lambda U^{\star,x}} \ind_{ \{ t < U^{\star,x} \} } = e^{-\lambda[U^{\star}(Z^x_{t+\cdot})+t ]} \ind_{ \{ t < U^{\star,x} \} }$.
 As a result, using \eqref{E:proof_Dynkin1} with the Markov property at time $t$, we deduce
\begin{equation} \label{E:representation_Mt_1}
M_t \ind_{ \{ t < U^{\star,x} \} }
=u_{\lambda}(Z^x_t)e^{-\lambda t} \ind_{\{t < U^{\star,x} \}}.
\end{equation}
 
We now consider 
$M_t \ind_{ \{ t \ge U^{\star,x} \} }$. We can write $M_t \ind_{ \{ t \ge U^{\star,x} \} }
=\mathbb{E} [e^{-\lambda U^{\star,x}} 
\mid \mathcal{G}_{t \wedge U^{\star,x}}] \ind_{ \{ t \ge U^{\star,x} \} }=
\mathbb{E} [e^{-\lambda U^{\star,x}} 
\mid \mathcal{G}_{ U^{\star,x}}] \ind_{ \{ t \ge U^{\star,x} \} }=
e^{-\lambda U^{\star,x}} \ind_{ \{ t \ge U^{\star,x} \} }
$. Using $u_\lambda(Z_{U^{\star,x}}^x)=1$ and \eqref{E:representation_Mt_1}, we deduce,
\begin{equation} \label{E:reecriture_M}
M_t=u_\lambda(Z^x_{t \wedge U^{\star,x}}) e^{-\lambda(t \wedge U^{\star,x})}.
\end{equation}
 The relation \eqref{E:reecriture_M} shows that $(u_\lambda(Z^x_{t \wedge U^{\star,x}}) e^{-\lambda(t \wedge U^{\star,x})})_{t \ge 0}$ is a martingale. 
 Recalling \eqref{E:EDS_Z_ell}--\eqref{E:def_compensator} and \eqref{E:exp_generateur_abstrait}, we apply Ito's formula to the process $t \mapsto 
 u_\lambda(Z^x_{t }) e^{-\lambda t}$ and find for $t < U^{\star,x}$,
 \begin{equation} \label{E:decompo_Doob}
 u_\lambda(Z^x_t)e^{-\lambda t}=
 u_\lambda(x)+
 \int_0^t e^{-\lambda s} [A u_\lambda (Z_s^x) - \lambda u_\lambda (Z_s^x)] \dd s
 + N_t,
 \end{equation}
where $N_t=\sum_{0 \le s \le t} [u_\lambda(Z_{s-}^x+\Delta_s Z^x)- u_\lambda(Z_{s-}^x)]
- \int_0^t  \int_0^\infty [ u_\lambda(Z_{s-}^x+a)- u_\lambda(Z_{s-}^x)] \nu(Z_s^x,\dd a) \dd s $ is a local martingale. By uniqueness of the Doob-Meyer decomposition, the predictive finite variation part is zero in the representation 
\eqref{E:decompo_Doob} and we deduce,
\begin{equation*}
 \int_0^t e^{-\lambda s} [A u_\lambda (Z_s^x) - \lambda u_\lambda (Z_s^x)] \dd s =0, \forall t<U^{\star,x}, \text{almost surely}.
 \end{equation*}
 From the almost sure continuity of $s \mapsto Z^x_s$ at zero, we finally obtain $A u_\lambda(x)=\lambda u_\lambda(x)$ for all $x>0$.
\end{proof}
\begin{rem}\label{R:mul_invariance}
In the proof of Lemma \ref{L:Dynkin} we have shown that the laws of the processes $t \mapsto Z^x_{t}$
and $t \mapsto x Z^1_{t/x}$ coincide until they reach zero. This is not surprising, since one can show using \eqref{E:def_compensator}
that the compensator of the point processes $t\mapsto Z^x_{t}+\beta_1 t $ and 
$t \mapsto x (Z^1_{t/x} + \beta_1 t/x)$ are the same. For point processes with finite intensities, it is known that the compensator characterizes the law of the process (see Theorem 1.26 in Chapter III of \cite{JacShi_Book}).
\end{rem}
We now show that the integro-differential equation \eqref{E:dynkin} can be transformed into
an ordinary differential equation. Related techniques were used in \cite{CheLeeShe07} for computing the ruin time
of Levy processes. 
In \cite{CheLeeShe07}, a crucial fact is that the generator of a Levy process acts as a multiplier in the Fourier domain. Such simplifications in the Fourier domain do not occur for the generator of the process
$Z^x$, however the multiplicative invariance of the process (see Remark \ref{R:mul_invariance}) suggests the use of the Mellin's transform. 
\begin{lem}\label{L:ODE}
The function $x \mapsto u_\lambda(x)$ is solution to
\begin{equation*}
\beta_1 x u_\lambda^{''}(x)+ u_\lambda^{'}(x)(\lambda x + \beta_1 \xi^\star)
-\lambda (\gamma -2) u_\lambda(x)=0, \text{ for all  $x\in (0,\infty)$,}
\end{equation*}
where $\xi^\star=\frac{1}{2\beta_1}-\frac{1}{2\beta_2}$ and the constant
$\gamma$ was defined in Proposition \ref{P:EDS_Z}.
\end{lem}
\begin{proof} The main idea is that the generator $A$ acts as a kind of multiplier for the
Mellin's transform.
Let us recall that for $f: [0,\infty) \to \mathbb{R}$ one defines the Mellin's transform of $f$ as
\begin{equation*} 
\mathcal{M}\left[ f \right] (\xi)
= \int_0^\infty x^{\xi-1} f(x) \dd x,
\end{equation*}
for all $\xi \in \mathbb{C}$ such that the latter integral is well defined. It is clear that if $f$
is bounded and with exponential decay near $\infty$ then $\xi \mapsto \mathcal{M}\left[ f \right](\xi)$ is well
defined and holomorphic on the half plane $\{ \xi  \in \mathbb{C} \mid \mathop{Re} (\xi)>0 \}$. 

For such functions $f$, we recall the four following properties which are easily derived from the definition of the Mellin's transform:
\begin{align} \label{E:pte_Mellin_1}
\mathcal{M}\left[ x \mapsto f(x(1+y)) \right](\xi) &=(1+y)^{-\xi} \mathcal{M} \left[ f \right](\xi),
\text{ for $\mathop{Re}(\xi)>0$, }
\\ \label{E:pte_Mellin_2}
\mathcal{M}\left[ x \mapsto f(x)/x \right](\xi) &= \mathcal{M}\left[ f \right](\xi -1),
\text{ for $\mathop{Re}(\xi)>1$, }
\\ \label{E:pte_Mellin_3}
\mathcal{M}\left[ f' \right](\xi)&=(1-\xi)\mathcal{M}\left[ f \right](\xi -1) ,
\text{ if $f \in \mathcal{C}^1(0,\infty)$ and $\mathop{Re}(\xi)>1$, }
\\ \label{E:pte_Mellin_4}
\mathcal{M}\left[ x \mapsto xf'(x) \right](\xi)&=-\xi \mathcal{M}\left[ f \right](\xi),
\text{ if $f \in \mathcal{C}^1(0,\infty)$ and $\mathop{Re}(\xi)>0$.}
\end{align} 
Now, using the expression of the generator \eqref{E:exp_generateur} with a simple change of variable we have
\begin{equation*}
A f (x)= -\beta_1 f'(x)+
\int_0^\infty [f(x(1+y))-f(x)] \frac{\kappa}{x} (1+y)^{-\gamma} \dd y.
\end{equation*}
Using Fubini's theorem and the properties \eqref{E:pte_Mellin_1}--\eqref{E:pte_Mellin_3} 
we deduce for $\mathop{Re}(\xi)>1$,
\begin{align*}
\mathcal{M}\left[ A f \right] (\xi)
&=\beta_1 (\xi-1) \mathcal{M}\left[f\right](\xi-1)
+ 
\left(
\kappa \int_0^\infty [(1+y)^{-\xi+1}-1](1+y)^{-\gamma} \dd y 
\right)
\mathcal{M}\left[f\right](\xi-1) 
\\
&=\left[  \beta_1 (\xi-1) + \frac{\kappa}{\xi+\gamma-2}-\frac{\kappa}{\gamma-1}
\right]
\mathcal{M}\left[f\right](\xi-1)
\\
&=Q(\xi) \mathcal{M}\left[f\right](\xi-1)
\end{align*}
with $Q(\xi)$ being the rational function $\ds Q(\xi)=\frac{\beta_1 (\xi-1) (\xi-\xi^\star)}{(\xi+\gamma-2)}$ and $\xi^\star=\frac{1}{2\beta_1}-\frac{1}{2 \beta_2}$.

Now we turn back to the solution of the equation $A u_\lambda = \lambda u_\lambda$ and apply the Mellin's transform on both sides of this equality. We deduce,
\begin{equation*}
Q(\xi) \mathcal{M}\left[ u_\lambda \right](\xi-1)
=
\lambda \mathcal{M}\left[ u_\lambda \right](\xi), \quad \forall \xi \text{ with } \mathop{Re}(\xi)>1.
\end{equation*}
From the definition of $Q$, we obtain,
\begin{equation*}
\beta_1 (\xi-1) (\xi-\xi^\star) \mathcal{M}\left[ u_\lambda \right] (\xi-1)
=
\lambda(\xi+\gamma-2)
\mathcal{M}\left[ u_\lambda \right] (\xi),
\end{equation*}
and using \eqref{E:pte_Mellin_3}--\eqref{E:pte_Mellin_4} this equation can be transformed into
\begin{equation*}
- \beta_1 \xi \mathcal{M}\left[u^{'}_\lambda(x) \right](\xi)
+ \beta_1 \xi^\star \mathcal{M}\left[ u^{'}_{\lambda}\right](\xi)
=
\lambda (\gamma -2) \mathcal{M}\left[u_\lambda \right](\xi)
-\lambda \mathcal{M}\left[ x \mapsto x u^{'}_\lambda (x) \right]
(\xi), \quad \mathop{Re}(\xi)>1.
\end{equation*}
We apply again \eqref{E:pte_Mellin_4} with the choice $f=u_\lambda'$. Remark that, even if $f=u_\lambda'$ is not bounded near $0$, it is easy to see that the property \eqref{E:pte_Mellin_4} is still valid, using 
$u'_\lambda(x) \overset{x\to 0}{=}O(1/x)$ and $\mathop{Re}(\xi)>1$. 
We deduce the following relation for all $\xi$ with $\mathop{Re}(\xi)>1$,
\begin{equation} \label{E:Mellin_final}
\beta_1 \mathcal{M}\left[ x \mapsto x u^{''}_\lambda(x) \right](\xi)
+ \beta_1 \xi^\star \mathcal{M}\left[ u^{'}_{\lambda}\right](\xi)
=
\lambda (\gamma -2) \mathcal{M}\left[u_\lambda \right](\xi)
-\lambda \mathcal{M}\left[ x \mapsto x u^{'}_\lambda (x) \right]
(\xi).
\end{equation}
Since the equality \eqref{E:Mellin_final} holds true for any $\xi $ in the half plane
$\mathop{Re}(\xi)>1$, we may invert the Mellin's transform, and we deduce the lemma.
\end{proof}
\subsection{Proof of Theorem  \ref{T:law_U}} \label{S:proof_th_lawU}
By Lemma \ref{L:ODE}, the function $u_\lambda(x)=\mathbb{E}_x[e^{-\lambda U^\star}]$ is solution to the equation $\beta_1 x u_\lambda^{''}(x)+ u_\lambda^{'}(x)(\lambda x + \beta_1 \xi^\star)
-\lambda (\gamma -2) u_\lambda(x)=0$,  for all  $x\in (0,\infty)$. Then, if we define 
\begin{equation}\label{E:def_w_lambda}
w_\lambda(x)=\frac{\lambda}{\beta_1} u_\lambda(-\frac{x \beta_1}{\lambda}) \text{ for $x<0$},
\end{equation}
 it is simple to check that $w_\lambda$ is solution to the Kummer's equation
\begin{equation}\label{E:Kummer}
x w^{''}_\lambda(x)+w^{'}_\lambda(x) [ \xi^\star -x ]+(\gamma-2) w_\lambda(x)=0,
\quad \text{ for all $x \in (-\infty,0)$.}
\end{equation} 
Moreover, from Lemma \ref{L:Dynkin}, this solution $w_\lambda$ satisfies the boundary condition
$\lim_{x \to 0} w_\lambda(x)=\frac{\lambda}{\beta_1}$ and $\abs{w_\lambda(x)}\le \frac{\lambda}{\beta_1} e^{-\abs{x}}$ for $x<0$.

We know from \cite{AbrSte_Book} that the Kummer's equation \eqref{E:Kummer} admits two independent solutions
\begin{align} \label{E:y1_Kummer}
y_1(x)&=M(2-\gamma,\xi^\star,x)
\\ \label{E:y2_Kummer}
y_2(x)&=e^x U(\xi^\star+\gamma-2,\xi^\star,-x)
=\frac{1}{\Gamma(\xi^\star+\gamma-2)}
\int_1^\infty e^{x t}  t ^{1-\gamma} (t-1)^{\xi^\star+\gamma-3} \dd t
\end{align}
where $M$ and $U$ are the confluent hypergeometric functions (the functions $M$ and $U$ are defined by formulas 13.1.2 and 13.1.3 in chapter 13 of \cite{AbrSte_Book}, and see formula 13.2.6 of \cite{AbrSte_Book} for the integral representation of $U$).
The asymptotic behaviour of the fundamental solutions can be found, using equation 13.1.5 in \cite{AbrSte_Book}
\begin{equation*}
y_1(x)\sim_{x \to -\infty} \frac{\Gamma(\xi^\star)}{\Gamma(\xi^\star+\gamma-2)}(-x)^{\gamma-2}
\end{equation*}
and using equation 13.5.2 in \cite{AbrSte_Book},
\begin{equation*}
y_2(x)\sim_{x \to -\infty} e^x (-x)^{2-\gamma-\xi^\star}.
\end{equation*}
From the exponential decay of $w_\lambda$, we deduce that $w_\lambda$ is proportional to $y_2$. Hence, using 
\eqref{E:y2_Kummer} we get
\begin{equation*}
w_\lambda(x)=c y_2(x)=\frac{c}{\Gamma(\xi^\star+\gamma-2)}
\int_1^\infty e^{x t} (t-1)^{\xi^\star+\gamma-3} t ^{1-\gamma} \dd t
\end{equation*}
for some $c\in \mathbb{R}$.

The condition $\beta_1 > \frac{\beta_2}{1+2\beta_2}$, equivalent to $\xi^\star<1$, is sufficient 
for the finiteness of  $U(\xi^\star+\gamma-2,\xi^\star,0)$ and 
$U(\xi^\star+\gamma-2,\xi^\star,0)= \frac{\Gamma(1-\xi^\star)}{\Gamma(\gamma-1)}$ (see
formulas 13.5.10--13.5.12 in \cite{AbrSte_Book}). We deduce that $c=\frac{\Gamma(\gamma-1)}{\Gamma(1-\xi^\star)} \frac{\lambda}{\beta_1}$ and
\begin{equation} \label{E:calcul_w}
w_\lambda(x)=\frac{\Gamma(\gamma-1)}{\Gamma(1-\xi^\star) \Gamma(\xi^\star+\gamma-2)}
\frac{\lambda}{\beta_1}
\int_1^\infty e^{x t} (t-1)^{\xi^\star+\gamma-3} t ^{1-\gamma} \dd t.
\end{equation}
From \eqref{E:def_w_lambda} and \eqref{E:calcul_w} we deduce,
\begin{align*}
\mathbb{E}_x[e^{-\lambda U^\star}]=u_\lambda(x)&=
\frac{\Gamma(\gamma-1)}{\Gamma(1-\xi^\star) \Gamma(\xi^\star+\gamma-2)}
\frac{\lambda}{\beta_1}
\int_1^\infty e^{-\frac{\lambda x t}{\beta_1}} (t-1)^{\xi^\star+\gamma-3} t^{1-\gamma} \dd t
\\
&=
\frac{\Gamma(\gamma-1)}{\Gamma(1-\xi^\star) \Gamma(\xi^\star+\gamma-2)}
\int_{x/\beta_1}^\infty 
e^{-\lambda u} \left( \frac{\beta_1 u}{x}-1 \right)^{\xi^\star+\gamma-3} 
\left(\frac{\beta_1 u}{x}\right)^{1-\gamma} \dd u.
\end{align*}
where in the last line we have made a simple change of variable.

Identification of the Laplace's transform shows that the law of $U^\star$ admits
the density $\frac{p_{U^\star}(x,\dd u)}{\dd u}= \frac{\Gamma(\gamma-1)}{\Gamma(1-\xi^\star) \Gamma(\xi^\star+\gamma-2)}
\left( \frac{\beta_1 u}{x}-1 \right)^{\xi^\star+\gamma-3} 
\left(\frac{\beta_1 u}{x}\right)^{1-\gamma} \ind_{\{ u \ge x/\beta_1 \}}$.
The expression \eqref{E:law_U_star} is obtained with simple algebra and recalling $\gamma=\frac{1+3\beta_2}{2\beta_2}$.
\qed

\section{Case of negative skewness coefficients}
\label{S:negative_beta}
In this section, we give sketches of the proofs of the results of Section \ref{S:Main} corresponding to the situations where one of the skewness parameters may be negative.

\subsection{Proof of Proposition \ref{P:coalescence_negative}}

If $X^{x,\beta}$ and $X^{x,\beta'}$ are two solutions of the skew Brownian motion equation \eqref{E:def_skew} with $-1<\beta <\beta'<1$ then $\mathbb{P}(X^{x,\beta}_t \le X^{x,\beta'}_t, \forall t \ge 0)=1$
(see Theorem 3.1 in \cite{BurChe01}).  
From this comparison property, it is simple to deduce Proposition \ref{P:coalescence_negative}
from Theorem \ref{T:coalescence}.

\subsection{Sketches of the proof of Theorem \ref{T:law_U_diff_sign}}

The proof of Theorem \ref{T:law_U_diff_sign} follows the same route as the proof of Theorem 
\ref{T:law_U}: we first determine the dynamic of $Z^x$ and then characterize the solution
of the associated Dynkin equation with help of the Mellin's transform. Finally, we identify the Laplace transform of the law of $U^\star$ with an explicit solution of the Kummer's equation. Let us give some more details.

In the case $-1<\beta_2<0<\beta_1<1$, 
by a proof similar to the one of Proposition \ref{P:L0_somme_ell}, we can show that 
\begin{equation*}
L^{0}_{\tau_t}(X^{x})= \sum_{0<u\le t} \ell(X_{\tau_{u-}}^x,\textbf{e}_u),
\end{equation*}
where the law of the functional $\ell$ under the excursion measure is 
\begin{equation*}
\textbf{n}_{\beta_1} ( \ell(h,\textbf{e}) \ge a) = 
\frac{1-\beta_1}{2h}
\left( 1+ \frac{\beta_2 a}{h} \right)^{-\frac{1+\beta_2}{2\beta_2}} \ind_{[0,\frac{h}{\abs{\beta_2}}]}(a), \quad 
\forall a >0.
\end{equation*}

We deduce that the dynamic of the process 
$Z^x_t=X^{x}_{\tau_t}$ is as follows
\begin{align*}
Z_t^x &= x-\beta_1 t + \sum_{0<u\le t} \beta_2 \ell(Z_{u-}^x,\textbf{e}_u),
\\&=
x-\beta_1 t  - \int_{[0,t]\times(0,\infty)}
a \mu(\dd u, \dd a),
\end{align*}
where the compensator of the random measure $\mu(\dd u, \dd a)$ is  $\dd u \times \nu(Z^x_{u-},\dd a)$ with
\begin{equation*} 
\nu(h,\dd a)= \frac{\abs{\kappa}}{h^2} \left(1-\frac{a}{h}\right)^{-\gamma} \ind_{[0,h]}(a) \dd a,
\quad \kappa=\frac{(1-\beta_1)(1+\beta_2)}{4 \beta_2} \text{ and } \gamma=\frac{1+3\beta_2}{2\beta_2}.
\end{equation*}

Remark that both the drift and jumps of $Z^x$ are negative, yielding to an almost sure finite hitting time of the level $0$. 
Especially, it is clear that the support of $U^\star$ is included in $[0,\frac{x}{\beta_1}]$. The generator of the process $Z^x$ now writes,
\begin{equation*}
A f (h)=-\beta_1 f'(h)+
\int_0^h [ f(h-a)-f(h)] \frac{\abs{\kappa}}{h^2} \left( 1 - \frac{a}{h} \right)^{-\gamma} \dd a, 
\quad\text{ for $h>0$.}
\end{equation*}
Consider $u_\lambda(x)=\mathbb{E}[e^{-\lambda U^{\star,x}}]$ where we use the notations of the proof of
Lemma \ref{L:Dynkin}. Exactly as in Lemma \ref{L:Dynkin}, we can prove that $u_\lambda$ satisfies the Dynkin's 
formula $A u_\lambda=\lambda u_\lambda$.
Moreover, as in Lemma \ref{L:Dynkin}, we can check that $U^{\star,x}$ and $x U^{\star,1}$ have the same law and hence $u_\lambda(x)=\mathbb{E}[e^{-\lambda x U^{\star,1}}]$. As a result, the function 
$x\mapsto u_\lambda(x)$ is the Laplace transform of the law of a random variable with the compact support 
 $[0,\frac{\lambda}{\beta_1}]$.  
 
 Then, the integro-differential equation $A u_\lambda=\lambda u_\lambda$ can be transformed to
 an ordinary differential equation by applying the Mellin's transform as in the proof of Lemma \ref{L:ODE}. 
 One finds exactly the same equation as in the case of positive skewness coefficients,
 \begin{equation} \label{E:EDO_section5}
\beta_1 x u_\lambda^{''}(x)+ u_\lambda^{'}(x)(\lambda x + \beta_1 \xi^\star)
-\lambda (\gamma -2) u_\lambda(x)=0, \text{ for all  $x\in (0,\infty)$,}
\end{equation}
with $\xi^\star=\frac{1}{2\beta_1}-\frac{1}{2\beta_2}$ and $\gamma=\frac{1+3\beta_2}{2 \beta_2}$.

Let us stress that some additional technical difficulties arise for the application of the Mellin's transform when $\beta_2<0$. 
Indeed, contrary to the case of Section \ref{S:Hitting}, we cannot show the exponential decay of $u_\lambda$
(due to the different shape of the support of the law of $U^\star$).
In order to define and manipulate the Mellin's transform of $u_\lambda$ on some sufficiently large strip,
some preliminary bounds on the decay of $u_\lambda$ have to be established.
The Lemma \ref{L:borne_ulambda_negative} of the Appendix ensures that all the necessary computations are allowed.


 As in the Section \ref{S:proof_th_lawU}, if one sets  $w_\lambda(x)=\frac{\lambda}{\beta_1} u_\lambda(-\frac{x \beta_1}{\lambda})$ for $x<0$, then $w_\lambda$ is solution to the Kummer's equation \eqref{E:Kummer}. A fundamental system of solution to the Kummer's equation is given by $y_1$ and $y_2$ (see \eqref{E:y1_Kummer}--\eqref{E:y2_Kummer}). For $\beta_2<0$, the solution $y_1$ admits an integral representation (see Formula 13.2.1 in \cite{AbrSte_Book}),
 \begin{equation} \label{E:y2_integral}
 y_1(x)=\frac{\Gamma(\xi^\star)}{\Gamma(\xi^\star + \gamma -2)\Gamma(2-\gamma)}
 \int_0^1 e^{xt} t^{1-\gamma} (1-t)^{\xi^\star+\gamma-3} \dd t.
 \end{equation}
Comparing \eqref{E:y2_Kummer} and \eqref{E:y2_integral}, with the fact that $u_\lambda$ is the Laplace transform of some function with compact support, we deduce that $w_\lambda$ is proportional to $y_1$. From the condition $w_\lambda(0)=\frac{\lambda}{\beta_1}u_\lambda(0)=\frac{\lambda}{\beta_1}$, we get,
\begin{equation}
 w_\lambda(x)= \frac{\lambda}{\beta_1}
\frac{\Gamma(\xi^\star)}{\Gamma(\xi^\star + \gamma -2)\Gamma(2-\gamma)}
 \int_0^1 e^{xt} t^{1-\gamma} (1-t)^{\xi^\star+\gamma-3} \dd t.
\end{equation}
Using $\mathbb{E}_{x}[e^{-\lambda U^\star}]= u_\lambda(x)=
\frac{\beta_1}{\lambda} w_\lambda(- \frac{\lambda x}{\beta_1} )
$
with a few computations, one can deduce \eqref{E:law_U_star_beta_neg}.

\subsection{Proof of Corollary \ref{C:x_neg} $(x>0, \beta_1<0, \beta_2<0)$}
Set $\widetilde{X}^{-x}=-X^x$, $\widetilde{X}^0=-X^0$ and $\widetilde{B}=-B$, then
\begin{align*}
\widetilde{X}^{-x}_t&=-x + \widetilde{B}_t + \abs{\beta_2} L^0_t(\widetilde{X}^{-x}),\\
\widetilde{X}^0_t &= \widetilde{B}_t + \abs{\beta_1} L^0_t(\widetilde{X}^0),
\end{align*}
so that we are now dealing with positive skewness coefficients, but a negative starting value $-x$.
Denote $T^0= \inf \{t >0 \mid \widetilde{X}^{-x}_t =0\}$. We define $\widehat{X}^0_t=\widetilde{X}^{-x}_{T_0+t}$, $\widehat{X}_t=\widetilde{X}^0_{T^0+t}$ and $\widehat{B}_t=\widetilde{B}_{T^0+t}-\widetilde{B}_{T^0}$. These processes
are solutions to
\begin{align*}
\widehat{X}_t&=\widehat{X}_0 + \widehat{B}_t + \abs{\beta_1} L^0_t(\widehat{X}),\\
\widehat{X}^0_t &= \widehat{B}_t + \abs{\beta_2} L^0_t(\widehat{X}^0),
\end{align*}
where $\widehat{X}_0$ is independent of $(\widehat{B}_t)_{t\ge 0}$.
Note that for these new processes the role of the skewness parameter has been exchanged, and the starting point
of $\widehat{X}$ is a positive random variable. Let us introduce $\widehat{T}^\star=\inf \{ t \ge 0:
\widehat{X}_t= \widehat{X}^0_t \}=T^\star-T^0$.

Using the Markov property at the random time $T^0$, and applying Theorem \ref{T:law_U}, we get that
$\displaystyle \mathbf{B}_1=\left(  \frac{ \abs{\beta_2} L_{\widehat{T}^\star}^0(\widehat{X}^0)}{\widehat{X}_0}    \right)^{-1}$ is independent of $\widehat{X}_0$ and distributed as a Beta $\mathcal{B} \left(1-\frac{1}{2\abs{\beta_2}}+\frac{1}{2\abs{\beta_1}}, \frac{1-\abs{\beta_2}}{2\abs{\beta_2}} \right)$ variable.

The random variable $\mathbf{B}_1$ can be related to the local time of the initial process,
\begin{align*}
L^0_{T^\star}(X^0)&=L_{T^\star}^0(\widetilde{X}^0)
\\
&=\frac{\abs{\beta_2}}{\abs{\beta_1}} 
L_{T^\star}^0(\widetilde{X}^{-x})-\frac{x}{\abs{\beta_1}}
\\
&
=\frac{\abs{\beta_2}}{\abs{\beta_1}} L_{\widehat{T}^\star}^0(\widetilde{X}^{-x}_{T^0+\cdot})
-\frac{x}{\abs{\beta_1}}=
\frac{\abs{\beta_2}}{\abs{\beta_1}} L_{\widehat{T}^\star}^0(\widehat{X}^0)
-\frac{x}{\abs{\beta_1}}
\\
&=\frac{\widehat{X}_0 \mathbf{B}_1^{-1}}{\abs{\beta_1}}
- \frac{x}{\abs{\beta_1}}
\\
&=\frac{\widetilde{X}^0_{T^0} \mathbf{B}_1^{-1}}{\abs{\beta_1}}
- \frac{x}{\abs{\beta_1}}
\\
&=\frac{[x+ \abs{\beta_1} L^{0}_{T^0}(\widetilde{X}^0) ] \mathbf{B}_1^{-1}}{\abs{\beta_1}}
- \frac{x}{\abs{\beta_1}}.
\end{align*}
But the law of $L^{0}_{T^0}(\widetilde{X}^0)$ may be derived by computations similar to
those of the step 3 in the proof of Proposition \ref{P:L0_somme_ell} (see \eqref{E:W_step3}). 
One finds that
$\mathbb{P}(L^0_{T^0}(\widetilde{X}^0) > a)=
\left( 1+ \frac{\abs{\beta_1} a}{x} \right)^{-\frac{1+\abs{\beta_2}}{2\abs{\beta_2}}}$.
This means that $\mathbf{B}_2=[1+\frac{\abs{\beta_1} L^0_{T^0}(\widetilde{X}^0) }{x}]^{-1}$ is distributed
as a $\mathcal{B}(\frac{1+\abs{\beta_1}}{2\abs{\beta_1}},1)$ variable.
Since $L^0_{T^\star}(X^0)= \frac{x}{\abs{\beta_1}}[ \mathbf{B}_2^{-1}\mathbf{B}_1^{-1}-1]$ the corollary is proved.

\section{Appendix}
We were unable to find a reference for the Kronecker Lemma in the context of continuous time local martingale defined on some random interval $[0,U]$ (however see \cite{Elliot01} for close results). Hence we give below a short proof of the result.

\begin{lem} 
\label{L:Kronecker}
Let $(\widetilde{\mathcal{J}}_t)_{0 \le t < U}$ be a locally square integrable ($\mathcal{G}_t$)-martingale with 
localizing sequence
$\tau_n=\inf \{ u \mid \left\langle \widetilde{\mathcal{J}},\widetilde{\mathcal{J}}  
\right\rangle_u \ge n \}$ and $U=\sup_{n} \tau_n$ (especially
 $U=\infty$ if $\left\langle \widetilde{\mathcal{J}},\widetilde{\mathcal{J}}  \right\rangle_\infty<\infty$). 
Assume additionally that $(\left\langle \widetilde{\mathcal{J}},\widetilde{\mathcal{J}}  \right\rangle_t)_{t \in [0,U)}$ is a continuous process.
Then, 

\begin{itemize}
\item On the set $\left\langle \widetilde{\mathcal{J}},\widetilde{\mathcal{J}}  \right\rangle_U<\infty$, we have the convergence
of $\widetilde{\mathcal{J}}_t$ as $t \to U$.
\item On the set $\left\langle \widetilde{\mathcal{J}},\widetilde{\mathcal{J}}  \right\rangle_U=\infty$ we have 
$\frac{\widetilde{\mathcal{J}}_t}{\left\langle \widetilde{\mathcal{J}},\widetilde{\mathcal{J}}  \right\rangle_t} \xrightarrow{t \to U} 0$.
\end{itemize}
\end{lem}
\begin{proof}
First, define the event 
$\Omega_{n}= \{\omega \mid \left\langle \widetilde{\mathcal{J}},\widetilde{\mathcal{J}} 
 \right\rangle_{\infty} < n \}$. On this event, $U=\tau_n=\infty$ and 
$(\widetilde{\mathcal{J}}_t)_{t \ge 0}=\widetilde{\mathcal{J}}_{t\wedge \tau_n})_{t \ge 0}$ is a bounded $\mathbf{L}^2$ 
martingale and thus converges as $t \to \infty$.
Since the convergence holds on the set $\Omega_{n}$ for all $n$, it holds on the set $ \{ \omega \mid \left\langle \widetilde{\mathcal{J}},\widetilde{\mathcal{J}}  \right\rangle_{\infty} <\infty\}$.

We now focus on the set $\{\left\langle \widetilde{\mathcal{J}},\widetilde{\mathcal{J}}  \right\rangle_{U}=\infty \}$. Define $\tilde{N}_{u}^n=
\int_0^{u\wedge \tau_n }  \frac{\dd \widetilde{\mathcal{J}}_u}{1+ \left\langle \widetilde{\mathcal{J}},\widetilde{\mathcal{J}}  \right\rangle_u }$ which is a $\mathbf{L}^2$-bounded martingale for each $n$ with,
$$
\left\langle \widetilde{N}^n,\widetilde{N}^n  \right\rangle_{u}
=\int_0^{u \wedge \tau_n} \frac{\dd \left\langle \widetilde{\mathcal{J}},\widetilde{\mathcal{J}} \right\rangle_u}{(1+ \left\langle \widetilde{\mathcal{J}},\widetilde{\mathcal{J}}  \right\rangle_u)^2 }
= 1 - \frac{1}{1+ \left\langle \widetilde{\mathcal{J}},\widetilde{\mathcal{J}}  \right\rangle_{u \wedge \tau_n} } \le 1.
$$
From this, we easily see that the sequence 
$( \int_0^{\tau_n} \frac{\dd \widetilde{\mathcal{J}}_u}{1+ \left\langle \widetilde{\mathcal{J}},\widetilde{\mathcal{J}}  \right\rangle_u } )_{n \ge 1}$ is a 
$\mathcal{G}_{\tau_n}$-martingale sequence which converges to some 
$\mathbf{L}^2$ variable 
$\int_0^U \frac{\dd \widetilde{\mathcal{J}_u}}{1+ \left\langle \widetilde{\mathcal{J}},\widetilde{\mathcal{J}}  \right\rangle_u }$. 
As a consequence $\widetilde{N}_u=\int_0^{u \wedge U}\frac{\dd \widetilde{\mathcal{J}}_u}{1+ \left\langle \widetilde{\mathcal{J}},\widetilde{\mathcal{J}}  \right\rangle_u }$ is a true $\mathbf{L}^2$-martingale. 

We now write for $u_0<u<U$,
\begin{align*}
\widetilde{\mathcal{J}}_u=\int_0^u (1+ \left\langle \widetilde{\mathcal{J}},\widetilde{\mathcal{J}}  \right\rangle_u) \dd \widetilde{N}_u
&=
\int_0^{u_0} (1+ \left\langle \widetilde{\mathcal{J}},\widetilde{\mathcal{J}}  \right\rangle_u) \dd \widetilde{N}_u
+
\int_{u_0}^{u} (1+ \left\langle \widetilde{\mathcal{J}},\widetilde{\mathcal{J}}  \right\rangle_u) \dd \widetilde{N}_u
\\
&=
\int_0^{u_0} (1+ \left\langle \widetilde{\mathcal{J}},\widetilde{\mathcal{J}}  \right\rangle_u) \dd \widetilde{N}_u
-\int_{u_0}^u (\widetilde{N}_v - \widetilde{N}_{u_0}) 
\dd \left\langle \widetilde{\mathcal{J}},\widetilde{\mathcal{J}}  \right\rangle_u
\\
& \quad \quad \quad
+
(1+\left\langle \widetilde{\mathcal{J}},\widetilde{\mathcal{J}}  \right\rangle_u)(\widetilde{N}_u-\widetilde{N}_{u_0})
\end{align*}
where we have used Ito's formula.

Then,
a convenient choice of $u_0$ determined by the almost sure convergence of $\widetilde{N}_u$ as $u\to U$, with 
$\left\langle \widetilde{\mathcal{J}},\widetilde{\mathcal{J}}  \right\rangle_U=\infty$, easily implies
that $\ds \frac{\widetilde{\mathcal{J}}_u}{\left\langle \widetilde{\mathcal{J}},\widetilde{\mathcal{J}}  \right\rangle_u} \xrightarrow{u \to U}0$.
\end{proof}

In the proof of Theorem  \ref{T:coalescence} we used the following lemma.
\begin{lem}
\label{L:eq_explosive}
No locally bounded measurable function $h : [0,\infty )\rightarrow {\mathbb R}^{\ast +}$ exist such that there exists $c>0$ and $t_0 \geq 0$ satisfying:
$\forall t\geq t_0$,
$$
\ln \left( h(t)\right) \leq -c \int_{t_0}^t \frac{\dd u}{h(u)}.
$$
\end{lem}
\begin{proof}
Set $g(t):= - \ln \left(h(t)\right)$. The inequality becomes $\displaystyle g(t) \geq c \int_{t_0}^t e^{g(u)}\dd u$. 
Denote $y(t)=\int_{t_0} ^{t} e^{g(u)}\dd u$ which is an increasing function.
One has the inequality between Stieljes measures on $[t_0,\infty)$, 
$\dd (e^{-c y(t)}) \le  - c \dd t$, that integrates to
$e^{-c y(t)}-e^{-c y(t_0)} \le - c (t-t_0)$. This yields to a contradiction as $t \to \infty$.
\end{proof}

\begin{lem}\label{L:borne_ulambda_negative}
Assume $-1<\beta_2<0<\beta_1<1$ and set $\rho$ any real number with $0<\rho< (1-\gamma)\wedge 1$ 
(recall that $\gamma=\frac{1+3\beta_2}{2\beta_2}$ is defined in Proposition \ref{P:EDS_Z}). 

Then, the following functions are bounded on $[0,\infty)$: 
$x \mapsto u_\lambda(x)$,  $x \mapsto x \abs{u'_\lambda(x)}$,  $x \mapsto x^2 \abs{u''_\lambda(x)}$,
$x \mapsto x^{1+\rho} u_\lambda(x)$,
$x \mapsto x^{2 + \rho} \abs{u'_\lambda(x)}$,
and $x^{3-\varepsilon} \abs{u''_\lambda(x)}$ for $0<\varepsilon<1$.

Especially, this implies that $\mathcal{M}\left[ u_\lambda \right] (\xi)
= \int_0^\infty x^{\xi-1} u_\lambda(x) \dd x$ is well defined on the strip $0 < \mathop{Re} (\xi) <1+\rho$.
\end{lem}
\begin{proof}
We use the notations of the proof of Lemma \ref{L:Dynkin}. We have $u_\lambda(x)=\mathbb{E}[e^{-\lambda U^{\star,x}}]
=\mathbb{E}[e^{-\lambda x U^{\star,1}}]$, and thus, $u_\lambda^{(k)}(x)=
\mathbb{E}[(-\lambda U^{\star,1} )^k e^{-\lambda x U^{\star,1}}]$ for $k \ge 0$. This clearly implies that $x^{k} \abs{ u_\lambda^{(k)}(x)}$ is a bounded function. For $k=0,1,2$, we get that the first three functions in the statement of the lemma are bounded.

Remark now that $U^{\star,1}$ is almost surely greater than $U^1_1$, the first jump time of the process
$u \mapsto Z^x_u$. The law of $U^1_1$ is given by \eqref{E:lawU11} and one can easily check that
$E[ (U^1_1)^{-1+\varepsilon} ]<\infty$ for $\varepsilon>0$. We deduce that
$E[ (U^{\star,1})^{-1+\varepsilon} ]<\infty$ for $\varepsilon>0$.
As a consequence,
\begin{align}\nonumber
\abs{u^{(k)}_\lambda(x)}
&=\abs{ \mathbb{E}[(-\lambda U^{\star,1} )^k e^{-\lambda x U^{\star,1}}] } \le  \mathbb{E}
\left[
\frac{1}{(U^{\star,1})^{1-\varepsilon}}
\frac{( \lambda x U^{\star,1})^{k+1-\varepsilon}}{\lambda^{1-\varepsilon} x^{k+1-\varepsilon}}
e^{-\lambda x U^{\star,1}}
\right]
\\ \label{E:controle_u_infty}
& \le c x^{-(k+1-\varepsilon)} E[ (U^{\star,1})^{-1+\varepsilon} ]
\le c x^{-(k+1-\varepsilon)} 
\end{align}
for some constant $c$ independent of $x$.
Using $k=2$, this shows that $x^{3-\varepsilon} \abs{u''_\lambda(x)}$ is bounded.

It remains to prove the boundedness of $x \mapsto x^{1+\rho} u_\lambda(x)$ and
$x \mapsto x^{2+\rho} \abs{u'_\lambda(x)}$. Clearly, only a control for large values of $x$ is needed. 
However, this control requires some additional work. 

We start by proving that $u_\lambda(x) \le c x^{-1-\rho}$ for $x>1$. Using Dynkin's equation $\lambda u_\lambda(x)=A u_\lambda(x)$, we have
\begin{align} \nonumber
\lambda u_\lambda(x)
&=-\beta_1 u'_\lambda(x)
+ 
\int_0^x \frac{\abs{\kappa}}{x^2} \left( 1-\frac{a}{x}\right)^{-\gamma} [u_\lambda(x-a)-u_\lambda(x)] \dd a
\\
&= \label{E:Dynkin_beta_neg}
-\beta_1 u'_\lambda(x)
-\frac{\abs{\kappa} u_\lambda(x)}{x} \int_0^1 a^{-\gamma} \dd a + 
\int_0^1 \frac{\abs{\kappa}}{x} a^{-\gamma} u_\lambda(x a) \dd a
\end{align}
where we performed a change of variables in the last line. 
From \eqref{E:controle_u_infty} with $k=0$ and $k=1$, we see that the first two terms in the right hand side 
of \eqref{E:Dynkin_beta_neg} are bounded by $c x^{-2+\varepsilon}=O(x^{-1-\rho})$ if $\varepsilon$ is small enough. 
It remains to control the last term in \eqref{E:Dynkin_beta_neg}. We split the integral  $\int_0^1 \frac{\abs{\kappa}}{x} a^{-\gamma} u_\lambda(x a) \dd a$ into
\begin{equation*}
\frac{\abs{\kappa}}{x} \int_0^{1/x}  a^{-\gamma} u_\lambda(x a) \dd a
+
 \frac{\abs{\kappa}}{x} \int_{1/x}^1 a^{-\gamma} u_\lambda(x a) \dd a.
\end{equation*}
Using the control $u_\lambda(xa) \le 1$ on the first integral and $u_\lambda(xa) \le c (xa)^{-1+\varepsilon}$ on the second one, we get 
$\int_0^1 \frac{\abs{\kappa}}{x} a^{-\gamma} u_\lambda(x a) \dd a \le c x^{\gamma-2} + c x^{-2+\varepsilon}
\le c x^{-1-\rho}$. Collecting all terms, we have shown $u_\lambda(x) \le c x^{-1-\rho}$.

To finish the proof of the lemma, we need to establish $\abs{u'_\lambda(x)} \le c x^{-2-\rho}$. 
If one differentiates the relation  \eqref{E:Dynkin_beta_neg} and uses \eqref{E:controle_u_infty}, for $k=1$ and $k=2$,
and the control already obtained for
$\int_0^1 a^{-\gamma} u_\lambda(xa) \dd a$,
it can be shown,
\begin{equation*}
\lambda u'_\lambda(x)=\frac{\abs{\kappa}}{x} \int_0^1 a^{-\gamma+1} u'_\lambda(xa) \dd a  +O(x^{-2-\rho}).
\end{equation*}
Again the integral above can be splitted into $\int_0^{1/x} a^{-\gamma+1} u'_\lambda(xa)\dd a+\int_{1/x}^1 a^{-\gamma+1} u'_\lambda(xa)\dd a $. Using the control $\abs{u'_\lambda(xa)} \le c (xa)^{-1}$ on the first part, and 
$\abs{u'_\lambda(xa)} \le c (xa)^{-2+\varepsilon}$ on the second part, we get
\begin{equation*}
\frac{1}{x} \int_0^1 a^{-\gamma+1} u'_\lambda(xa) \dd a \le c x^{-3+ \gamma}+c x^{-3+\varepsilon}
\le c x ^{-(2+\rho)}.
\end{equation*}
Collecting all terms, we have shown $\abs{u'_\lambda(x)} \le c x^{-2-\rho}$.
\end{proof}

\bibliographystyle{plain}





\end{document}